\documentclass[12pt]{article}
\usepackage{graphicx}
\usepackage[utf8x]{inputenc}
\usepackage{amsmath,sectsty,amsthm,amssymb,fontenc,color,tensor,textcomp,amsfonts,relsize,graphicx,graphics,enumerate}
\usepackage{fancyhdr}
\usepackage[colorlinks=true,
            linkcolor=red,
            urlcolor=blue,
            citecolor=green]{hyperref}
\usepackage{mathtools}
\usepackage{amsmath,tikz}
\usepackage{tensor}
\usepackage[rmargin=3.5cm, lmargin=3.5cm]{geometry}
\newtheorem{theo}{Theorem}[section]
\newtheorem{lem}[theo]{Lemma}

\newtheorem{prop}[theo]{Proposition}
\newtheorem{cor}[theo]{Corollary}
\newtheorem{df}[theo]{Definition}
\newtheorem{rmk}[theo]{Remark}
\newtheorem{cl}[theo]{Claim}

\title{Regularity of four dimensional Willmore-type hypersurfaces }
\newcommand\norm[1]{\left\lVert#1\right\rVert}

\date{}
\author{Yann Bernard\footnote{School of Mathematics, Monash University, 3800 Victoria, Australia. Email address: yann.bernard@monash.edu}\quad and \, Peter Olamide Olanipekun\footnote{School of Mathematics, Monash University, 3800 Victoria, Australia. 
\hspace{3mm}Email address: olanipekunp@gmail.com}}
\begin{document}\sloppy

\maketitle
\abstract{A four dimensional conformally invariant energy is studied. This energy generalises the well known two-dimensional Willmore energy. Although not positive definite, it includes minimal hypersurfaces as critical points. We compute its first variation  and by applying Noether’s theorem to the invariances, we derive some conservation laws which are satisfied by its critical points and with good analytical dispositions. In particular, we show that its critical points are smooth. We also investigate other possible four dimensional generalisations of the Willmore energy, and give strong evidence that critical points of such energies do not include minimal hypersurfaces.}

\tableofcontents

\section{Introduction}

Let $\vec{\Phi}$ be an immersion of the surface $\Sigma$ into the Euclidean space $\mathbb{R}^m$. Let $g$ and $\vec{h}$  denote the first and second fundamental form on $\Sigma$ respectively.
Let $\vec{H}:= \frac{1}{2}\textnormal{Tr}_g\vec{h}$ and $\vec{h}_0:=\vec{h}-\vec{H} g$ \,\,be the mean curvature vector and the trace-free second fundamental form respectively. Denote by $\chi(\Sigma)$ the Euler characteristic of the surface $\Sigma$. Owing to the Gauss-Bonnet theorem, the Willmore energy can be written as
\begin{align}
\int_\Sigma|\vec{H}|^2 d\textnormal{vol}_g = \frac{1}{4}\int_\Sigma|\vec{h}|^2 d\textnormal{vol}_g +\pi\chi(\Sigma) =\frac{1}{2}\int_\Sigma|\vec{h}_0|^2 d\textnormal{vol}_g +2\pi\chi(\Sigma)  \label{apps}
\end{align}
    so that the three energies appearing in \eqref{apps} have the same critical points. Although the energy $\int_\Sigma|\vec{h}_0|^2 d\textnormal{vol}_g$ is conformally invariant, the other two energies in \eqref{apps} are not.
\noindent
 Willmore surfaces are the critical points of the Willmore energy and are found by varying the energy. They satisfy the following Euler-Lagrange equation which is known in literature as the Willmore equation: 
\begin{align}
\mathcal{\vec{W}}:=\Delta_\perp\vec{H} +(\vec{H}\cdot\vec{h}_{ij})\vec{h}^{ij} -2|\vec{H}|^2\vec{H}=\vec{0} \label{eqya}
\end{align}
where $\Delta_\perp$ is the negative covariant Laplacian for the Levi-Civita connection in the normal bundle. We will refer to the left hand side of \eqref{eqya} as the {\it Willmore invariant}. In codimension one, the above equation reduces to
\begin{align}
\Delta_g H+|h_0|^2 H=0  \label{eqya1}
\end{align}
where $\Delta_g$ is the negative Laplace-Beltrami operator on $\Sigma$.  Since the mean curvature depends on two derivatives of the immersion, the Willmore equation \eqref{eqya1} (or more generally \eqref{eqya}) is  a fourth order elliptic nonlinear partial differential equation. The Willmore energy has been widely studied in literature with recent efforts geared towards the analysis of Willmore surfaces under different conditions: in codimension one \cite{cdd,stru}, higher codimension \cite{noetherpaper}, in a fixed conformal class \cite{boh, sc} and under other frameworks \cite{yber2, ybr, yber1,  br, bry, riv} just to mention a few. Of importance is the Willmore conjecture (\cite{willmoretom}, see also \cite{lipeter}) which was proved in \cite{fcmarqu} by Fernando Cod\'a Marques and Andr\'e Neves. 

The main objects of this paper are the four dimensional Willmore energy and its critical points. We highlight  the following interesting properties about the Willmore energy which will guide our discussion on four dimensional generalizations of the Willmore energy.
\begin{enumerate}[(a).] \label{ajaka}
\item\label{ajaka1} The Willmore energy is conformally invariant up to Gauss-Bonnet terms.
\item \label{ajaku} The leading order operator in the Willmore equation is linear.
\item \label{ajaki} Minimal surfaces are critical points of the Willmore energy.
\item \label{ajaka2} The Willmore energy is non-negative.
\end{enumerate}

\subsection{A four dimensional  Willmore-type energy}
Generalising the Willmore energy to higher dimensions has been a question of interest for a long time. It has been often asked what properties or conditions must be satisfied by an energy that would qualify as a generalised higher dimensional version of the Willmore energy. 
This essentially leads us to ask if there is a four dimensional Willmore-type energy preserving the properties \eqref{ajaka1}-\eqref{ajaka2} above.

\noindent
By considering an energy involving a linear combination of the quantities
$$(\vec{h}_0)_{ij}\cdot (\vec{h}_0)^{jk}(\vec{h}_0)_{kl}\cdot (\vec{h}_0)^{li}\,, \quad (\vec{h}_0)_{ij}\cdot (\vec{h}_0)^{ij}(\vec{h}_0)_{kl}\cdot (\vec{h}_0)^{kl} \quad\mbox{and}\, ((\vec{h}_0)_{ij}\cdot (\vec{h}_0)^{jk}(\vec{h}_0)_{k}^i)^{4/3}$$
one can easily construct a four dimensional conformally invariant energy. However, preliminary computations show that such energies do not preserve the properties \eqref{ajaku} and \eqref{ajaki} above (see for instance those considered in \cite{rigoli} and in Section 5 of \cite{mondino}). Thus a different approach, which we shall now describe, is needed.

\vskip 3mm
\noindent
There is a relationship between the problem of finding a suitable higher dimensional generalisation of the Willmore energy and the singular Yamabe problem.  Let $(M^n, \bar g)$ be a smooth compact Riemannian manifold of dimension $n$ with boundary and let $\Sigma$ be an hypersurface, the singular Yamabe problem asks for a defining function $u$  for the boundary $\Sigma:=\partial M$ such that on the interior $ \mathring{M}$, the scalar curvature $R_g$ of the metric $g=u^{-2}\bar g$ satisfies   $R_g=-n(n-1)$.   
This problem was first considered by Loewner and Nirenberg in \cite{loewner}. Later, 
Andersson, Chr\'usciel and Friedrich \cite{acf} computed a conformal invariant which obstructs the smoothness of the function $u$. This obstruction to smooth boundary asymptotics for the Yamabe solution  was identified by Gover et.\ al as a Willmore surface invariant in $\mathbb{R}^3$ (see \cite{{GGHW}, {GoW2}, {GoW3}}). This identification led to the perception that higher dimensional generalisation of the Willmore invariant can be found  via obstructions to smooth boundary asymptotics of the singular Yamabe problem on conformally compact manifolds of higher dimension. They further asked if there is a corresponding energy for such invariants, that is, is there a higher dimensional generalisation of the Willmore energy whose Euler-Lagrange equation corresponds to the identified obstruction? It turns out that the answer is positive (see \cite{{govie}, {gragra}}). Moreover, the following four dimensional generalisation of the Willmore energy was identified
\begin{align}
\mathcal{E}(\Sigma):=\int_{\Sigma} (|\pi_{\vec{n}}\nabla\vec{H}|^2- |\vec{H}\cdot\vec{h}|^2 +7|\vec{H}|^4) \,\,\,d\textnormal{vol}_g  \label{jkas}
\end{align}
where notations have been slightly adjusted as follows: $\vec{H}:= \frac{1}{4}g^{ij}\vec{h}_{ij}=\frac{1}{4}\vec{h}^j_j$, \,\,\, $|\pi_{\vec{n}}\nabla\vec{H}|^2:=\pi_{\vec{n}}\nabla_i\vec{H}\cdot\pi_{\vec{n}}\nabla^i\vec{H}$, where $\pi_{\vec{n}}\nabla\vec{H} $ is the projection of the Levi-Civita connection, acting on $\vec{H}$, onto the normal bundle, \,\,\, $|\vec{H}\cdot\vec{h}|^2:= (\vec{H}\cdot\vec{h}_{ij})(\vec{H}\cdot\vec{h}^{ij})$ and $|\vec{H}|^4:= (\vec{H}\cdot\vec{H})^2$. This energy was first discussed  in codimension one\footnote{In codimension one, the energy $\mathcal{E}(\Sigma)$ reduces to $\int_{\Sigma} (|\nabla H|-|H|^2|h|^2+7|H|^4)\,\,\,d\textnormal{vol}_g$.} by Guven in \cite{jguv} and rediscovered by Robin-Graham and Reichert in \cite{robingraham}, and separately by Zhang in \cite{zhang}. The energy \eqref{jkas} satisfies the conditions \eqref{ajaka1}-\eqref{ajaki} specified above; infact, it satisfies \eqref{ajaka1} without Gauss-Bonnet. However, as noted in \cite{robingraham}, it is not bounded below, thus violating condition \eqref{ajaka2}. It is possible to construct a non-negative energy by adding, to the energy \eqref{jkas}, a multiple of the fourth power of the norm of the trace-free second fundamental form. In such situation, one considers the energy:
\begin{align}
\mathcal{E}^1(\Sigma):= \mathcal{E}(\Sigma)+\mu\int_\Sigma    |h_0|^4\,\,d\textnormal{vol}_g  \label{generalpo}
\end{align}
for some  $\mu \geq\frac{1}{12}$. But one discovers that the energy $\mathcal{E}^1(\Sigma)$ seems to violate the very crucial condition \eqref{ajaki} (see Appendix \ref{othergen}). Thus,  this paper is focused on studying the critical points (of which many examples are known) of the energy \eqref{jkas}.

\subsection{ Regularity Proof Strategy}

\subsubsection{General literature on the strategy in two dimensions}
The two dimensional Willmore energy is finite as soon as  the Willmore immersion $\vec{\Phi}:\Sigma \rightarrow \mathbb{R}^m$ belongs to the Sobolev space $W^{1,\infty}\cap W^{2,2}$.  This implies that $\vec{h}\in L^2$.  
 This fact, when incorporated  into the Willmore equation, gives no analytical meaning as the cubic curvature terms $(\vec{H}\cdot\vec{h})\vec{h}-2|\vec{H}|^2\vec{H}$ (or more importantly $\Delta_\perp\vec{H}$ as per \eqref{eqya}) would not even belong to $L^1$. However, this issue can be resolved with the help of conservation laws.
In the seminal article \cite{riv}, the author found such conservation laws. It was shown that the fourth-order Willmore equation, associated with the two-dimensional Willmore energy, can be expressed in divergence forms with several analytical implications. These divergence forms were elegantly obtained, in \cite{noetherpaper}, by applying Noether's theorem (a general result applicable in any dimension) to the invariances of the Willmore energy and several other pre-existing results were rediscovered. Indeed, there exist some  1-forms $\vec{V}_{\textnormal{tra}}$, ${V}_{\textnormal{dil}}$ and $\vec{V}_{\textnormal{rot}}$ such that

\noindent
\begin{align}
\vec{\mathcal{W}}=d^\star \vec{V}_{\textnormal{tra}}=\vec{0}\,\,,\quad d^\star {V}_{\textnormal{dil}}={0}\,\,,\quad d^\star \vec{V}_{\textnormal{rot}}=\vec{0}   \label{aaq12} \end{align}
where $\vec{\mathcal{W}}$ is the Willmore invariant. The quantity $\vec{V}_{\textnormal{tra}}$ is  geometric since it depends on $\vec{\mathcal{W}}$. The quantities ${V}_{\textnormal{dil}}$ and $\vec{V}_{\textnormal{rot}}$ are also geometric since they depend on $\vec{V}_{\textnormal{tra}}$.
\noindent
Applying Poincar\'e lemma to \eqref{aaq12}, one finds the primitives $\vec{L}$, $S$  and $\vec{R}$ satisfying
\[   \left\{
\begin{array}{ll}
      d^\star \vec{L}=\vec{V}_{\textnormal{tra}} \,\,,\quad d^\star {S}={V}_{\textnormal{dil}}\,\,,&d^\star\vec{R}= \vec{V}_{\textnormal{rot}} \\
      d \vec{L}=\vec{0} \,\,,\quad\quad\quad d {S}={0}\,\,,&d\vec{R}= \vec{0} \\
\end{array} 
\right. \]
where the Hodge star $\star$, the exterior differential $d$ and codifferential $d^\star$ are understood in terms of the induced metric $g$. 
Further computations \cite{noetherpaper} yield the following system\footnote{The geometric product $\bullet$ is defined in the next Section 
.} 
of second-order partial differential equations:
\begin{align}  \left\{
\begin{array}{ll}
      |g|^{1/2}\Delta_g (\star S)&=\star\big(d(\star\vec{n})\cdot d(\star\vec{R})\big) \\
      |g|^{1/2}\Delta_g(\star\vec{R})&=\star\big(d(\star\vec{n})\bullet d(\star\vec{R})\big) +\star\big(d(\star\vec{n})\,\, d(\star{S})\big)  \\
|g|^{1/2}\Delta_g\vec{\Phi}&=\star\big(d(\star\vec{R})\bullet d\vec{\Phi}\big) +\star\big(d(\star{S})\,\, d\vec{\Phi}\big)  \\
\end{array} 
\right. \label{sisio5} \end{align}
where $\vec{n}$ denotes the Gauss map on $\Sigma$. The system \eqref{sisio5} possesses a Jacobian structure suitable for analysis \cite{{ybr}, {noetherpaper},  {palais}, {riv}}.

\noindent
The question of regularity of the two dimensional Willmore surface is  already settled in literature. For instance, it was proved in \cite{palais} and \cite{riv} that  Willmore immersions associated with the two dimensional Willmore energy are smooth. The proof techniques heavily rely on the structure of the system \eqref{sisio5}. The regularity assumptions on $\vec{\Phi}$ imply that $\vec{n}$, $S$ and $\vec{R}$ belong to the Sobolev space $W^{1,2}.$ Thus it follows that $\Delta_g (\star S)$, $\Delta_g(\star\vec{R})$ and $\Delta_g\vec{\Phi}$ all belong to $L^1$; this information  is critical for regularity purpose. However, the Jacobian structure of system \eqref{sisio5} encodes an extra regularity. By using Wente estimates (see Chapter 3 of \cite{ helein}, \cite{palais}), one improves the regularity of  $(\star S)$ and $(\star\vec{R})$ by showing that they both belong to $W^{1,p}$ for some $p>2$.
The third equation in \eqref{sisio5} links the primitives $(\star S)$ and $(\star\vec{R})$ back to the Willmore immersion $\vec{\Phi}$. In fact, it follows that $\vec{\Phi}$ belongs to $W^{2,p}$. By applying a bootstrapping argument, one concludes that  $\vec{\Phi}$ is smooth.

 These formulations and results lead us to ask if analogous results can be obtained in four dimensions. Firstly, what conservation laws are satisfied by the critical points of the energy $\mathcal{E}(\Sigma)$?
Understanding  Willmore-type equations  in higher dimensions, from analytical perspective, can be generally a difficult problem. 
For instance, the Willmore-type equation associated with \eqref{jkas} is a sixth-order nonlinear partial differential equation, the study of which is rare in literature. 
Secondly,  one asks if the regularity proof procedure in two dimensions can be applied in four dimensions.  These questions are clearly answered in this paper. We develop a regularity proof for the critical points of the energy $\mathcal{E}(\Sigma)$.

\subsubsection{Difficulties of this paper and our proof strategy}
We begin with the statement of our main result.
\begin{theo}\label{heart}
Let $\vec{\Phi}\in W^{3,2}\cap W^{1,\infty}$ be a non-degenerate critical point of the energy
\begin{gather}
\mathcal{E}(\Sigma):= \int_\Sigma \left(|\pi_{\vec{n}}\nabla \vec{H}|^2-|\vec{H}\cdot\vec{h}|^2 +7|\vec{H}|^2    \right) d\textnormal{vol}_g.  \label{sjda}
\end{gather}
Then $\vec{\Phi}$ is smooth with the estimate
\begin{align}
\norm{DH}_{L^\infty(B_r)} +\frac{1}{r} \norm{H}_{L^\infty(B_r)}  \lesssim \frac{1}{r^2} \left(\norm{DH}_{L^2(B_1)} +\norm{H}_{L^4(B_1)}   \right) \quad\quad\forall \, r<1   \nonumber
\end{align}
where $B_r\subset\mathbb{R}^4$ is any ball of radius $r$.
\end{theo}

\noindent
It is interesting to note that the proof techniques employed in four dimensions  is remarkably different from the regularity proofs of similar problems in two dimensions  (such as the Willmore surfaces, harmonic maps of surfaces etc). Unlike the two dimensional case, the corresponding expressions for $\vec{L}$, $S$ and $\vec{R}$ are much more intricate. Also, the system of equations does not seem to possess Jacobian structure, hence we do not use Wente estimates in our proof. However with the help of a very deep Bourgain-Brezis type result and classical results of Miranda \cite{mir} and Di Fazio \cite{faz}, we are able to obtain preliminary results needed to begin the regularity proof. It is important to remark that one can bypass the lack of a Jacobian structure by applying Bourgain-Brezis result to conservation laws derived from the invariances of the energy.
This present paper is the first step towards the analytical study of the Willmore-type energy \eqref{jkas} which is new in literature.

Let $\Sigma$ be a four dimensional hypersurface and let $\Omega \subset \Sigma$ be a chart. The conformal invariance of $\mathcal{E}(\Sigma)$ gives rise to the existence of a 1-form $\vec{V}$ satisfying the following conservation laws
$$d^\star\vec{V}=\vec{0}\,,\quad\quad d^\star(\vec{\Phi}\cdot\vec{V}-d|\vec{H}|^2)=\vec{0}\, \quad\mbox{and}\quad d^\star(\vec{\Phi}\wedge\vec{V} +\vec{M})$$
where the quantities $\vec{V}$ and $\vec{M}$ depend on  the  derivatives of $\vec{\Phi}$, up to fifth order.
The Poincar\'e Lemma then yields the existence of vector valued 2-forms $\vec{L}_0, \vec{L}$ and $\vec{R}_0$ and a scalar valued 2-form $ S_0$  satsifying \footnote{ Notations used here will be clarified later in section \ref{latty1}}
\begin{align} \vec{V}= d^\star \vec{L}_0, \quad dS_0= \vec{L} \overset{.}\wedge_4 d\vec{\Phi}, \quad\mbox{and}\quad d\vec{R}_0=\vec{L}\overset{\wedge}\wedge_4 d\vec{\Phi} +\star d\vec{u}- \vec{G}\wedge d\vec{\Phi} \label{amusan} \end{align}
for suitable $\vec{u}$ and $\vec{G}$,  where $\vec{L}= \star(\vec{L}_0-d|\vec{H}|^2\wedge d\vec{\Phi})$. Notice that $S_0$ and $\vec{R}_0$ relates to $\vec{L}$, and hence to $\vec{V}$. Unlike in two dimensions, it seems quite difficult to glimpse the following type of system of equations with a Jacobian structure which encodes extra regularity information
\[   \left\{
\begin{array}{ll}
      \Delta S_0=\textnormal{Jacobian} (d(\star\vec{\eta}), d\vec{R}_0) \\
      \Delta \vec{R}_0= \textnormal{Jacobian} (d(\star\vec{\eta}), dS_0, d\vec{R}_0) \\
\end{array} 
\right. \]
 for a suitable $\vec{\eta}$. Hence the need for a different approach which makes this paper even more interesting and original. 
Since the energy $\mathcal{E}(\Sigma)$ is not bounded below, we define  a non negative finite energy $E(\Omega):= \norm{DH}_{L^2(\Omega)} +\norm{H}_{L^4(\Omega)} $ which will serve as the upper bound for some of our estimates. The steps are as follows.

\begin{itemize}
\item Regularity of $\vec{V}$, $\vec{L}_0$ and $\vec{L}$: the energy hypothesis $\vec{\Phi}\in W^{3,2}\cap W^{1,\infty}$ and the dual Sobolev injection $\dot{W}^{-1,4/3}\subset \dot{W}^{-2,2}$ imply that $\vec{V}\in \dot{W}^{-2,2}\oplus L^1$ with the estimate $\|\vec{V}\|_{\dot{W}^{-2,2}\oplus L^1} \lesssim E(\Omega) $. A Bourgain-Brezis result allows us to find $\vec{L}_0\in\dot{W}^{-1,2}$ which solves the first equation in \eqref{amusan} together with the estimate  $\|\vec{L}_0\|_{\dot{W}^{-1,2}(\Omega)} \lesssim E(\Omega)$. The latter and the fact that $\vec{L}= \star(\vec{L}_0-d|\vec{H}|^2\wedge d\vec{\Phi})$ then gives  $\vec{L}\in\dot{W}^{-1,2}$ with a similar estimate.

\item Regularity of $\vec{u}$, $S_0$ and $\vec{R}_0$: a classical result of Miranda \cite{mir} and some standard elliptic estimate helps us find the estimate $\|D\vec{u}\|_{L^2(\Omega')}\lesssim  E(\Omega)$ for all $\Omega'\Subset \Omega$. Next is to use the latter and the estimate for $\vec{L}$ in the second and third equations in \eqref{amusan} to obtain $dS_0, d\vec{R}_0\in \dot W^{-1,2}$ with the estimate $\|dS_0\|_{\dot W^{-1,2}(\Omega)} +\|d\vec{R}_0\|_{\dot W^{-1,2}(\Omega)} \lesssim E(\Omega_1)$ for all $\Omega \Subset \Omega_1$. Since $S_0$ and $\vec{R}_0$ are co-closed, elliptic regularity gives $S_0,\vec{R}_0\in L^2(\Omega)$.

\item The need for $S$ and $\vec{R}$ with estimates in the correct space: the $\dot W^{-1,2}$ estimates on $dS_0$ and $d\vec{R}_0$ are insufficient to bootstrap. We need $L^{4/3}$ estimates in place of $\dot W^{-1,2}$ estimates. Here is a major difficulty prompting a novel approach in our proof. Our strategy is to redefine $\vec{L}$, $S_0$ and $\vec{R}_0$ as $\vec{L}_1$, $S$ and $\vec{R}$ in such a way that these new 2-forms readily satisfy the differential identities analogous to those for $\vec{L}$, $S_0$ and $\vec{R}_0$. In order to do this, we introduce an operator which is algebraically invertible and whose inverse maps $L^2$ into $L^2$ continuously. With a bit of work, we obtain the crucial estimate
$$\|dS\|_{L^{4/3}(B_{kr})} + \|d\vec{R}\|_{L^{4/3}(B_{kr})} \lesssim (k+\varepsilon)E(B_r) \quad\quad \mbox{for} \,\,k\in(0,1/2).$$

\item Controlling geometry and bootstrapping: we need to relate the regularity information we have of the potentials $S$ and $\vec{R}$ back to a geometric quantity ($\vec{\Phi}$ or $\vec{H}$). A short computation gives a {\it return equation} that lets us return back to geometry. We insert the $L^{4/3}$ estimate on $dS$ and $d\vec{R}$ in the return equation and infer the estimate 
$$\|\Delta H\|_{L^{4/3}(B_{kr})} + \|d H\|_{L^2(B_{kr})} \lesssim (\varepsilon + k)E(B_r)  \quad\quad \mbox{for} \,\,k\in(0,1/2).$$
Sobelev embedding and standard elliptic estimates then yield
$E(B_{kr})\lesssim (\varepsilon +k)E(B_r)$ from which we obtain the Morrey decay $E(B_r)\lesssim r^\beta E(B_1)$ for $r<1$. This then allows us to obtain a bound for the maximal function of $\Delta \vec{H}$. We call upon a lemma of Adams \cite{ada} in order to bound $D\vec{H}$ in $L^s(B_r)$ when $s\in(2, 8/3)$, thus improving the regularity of $D \vec{H}$. This can then be bootstrapped until $\vec{H}$ is Lipschitz, consequently $\Delta \vec{\Phi}$ is as well Lipschitz.  Observe  that $|g|^{1/2}\vec{H}=|g|^{1/2}\Delta \vec{\Phi}  \in W^{1,4}$.  Lorentz space theory guarantees $D\vec{\Phi} \in C^{0,\alpha}$ for $\alpha < 1$.
One can then conclude that $\vec{\Phi}$ is smooth by standard elliptic theory.
\end{itemize}

\vskip 3mm
\noindent
The paper is organised as follows. In Section \ref{latty1} we compute the  {\it Willmore-type equation} for the energy $\mathcal{E}(\Sigma)$, this is basically Theorem \ref{WE}. Thereafter, we investigate other possible four dimensional generalisations of the energy $\mathcal{E}(\Sigma)$ and argue that the only energy whose critical points include minimal hypersurfaces is the energy $\mathcal{E}(\Sigma)$; the other energies violate the condition \eqref{ajaki}. We then proceed to obtain conservation laws for the energy $\mathcal{E}(\Sigma)$ in Theorem \ref{cons-law} and Proposition \ref{coro}. In Section \ref{chap4}, we prove the main result of the paper. This is done in a series of steps organised into Subsections \ref{adzsq} - \ref{linkage}.   
\vskip 3mm

\noindent
{\bf Notational conventions.} Unless otherwise specified, we will adopt the following notations and conventions in this work. 
\begin{enumerate}
\item We will use the Greek letter $\delta$ to denote the variation of a quantity. The Greek letter $\delta$ with indices will be reserved for the component of the flat metric. For example, $\delta g_{ij}$ is the variation of the $(i,j)$ component of the metric $g$; while $\delta^{i}_j $ is the $(i,j)$ component of the flat metric.

\item We adopt Einstein's summation convention for our computations. 

\item We denote by $\pi_{\vec{n}} (\vec{V})$ and $\pi_{T}(\vec{V})$ the normal and tangential projections respectively of some vector $\vec{V}$. We have $\pi_{\vec{n}}+\pi_{T}=\textnormal{id}$.

\item Also, we use the metric tensor $g_{ij}$  (or the inverse metric  tensor $g^{ij}$) to lower (or raise)  indices on tensors. For example, $\tensor{T}{_j_l}= g_{ij}g_{kl} T^{ik}$ and $\tensor{T}{^j^l}=g^{ij}g^{kl}T_{ik}$.

\item The operators $\Delta$, $\Delta_g$ and $\Delta_0$ will denote the Hodge Laplacian, the Laplace Beltrami and the flat Laplacian respectively. The operators $\Delta$ and $\Delta_g$ coincide when they act on ``scalar" functions.

\item
We now briefly explain some of the notations and conventions we have adopted. The {\it small dot operator} $\cdot$ is used to denote the usual dot product between vectors. We also use the small dot to denote the natural extension of the dot product in $\mathbb{R}^m$ to multivectors (see \cite{fed}). We define the bilinear map 
$\mathlarger{\mathlarger{\mathlarger{\mathlarger{\llcorner}} }} :\Lambda ^p(\mathbb{R}^m)\times \Lambda^q(\mathbb{R}^m)\rightarrow \Lambda^{p-q}(\mathbb{R}^m)$, known as the {\it interior multiplication} of multivectors, by
$$
\langle \vec{A}\,\,\mathlarger{\mathlarger{\mathlarger{\mathlarger{\llcorner}} }} \vec{B},\, \vec{C}\rangle_{\mathbb{R}^m}=\langle \vec{A}, \vec{B}\wedge \vec{C}\rangle_{\mathbb{R}^m} \label{se123}$$
where $\vec{A}\in \Lambda^p(\mathbb{R}^m)$, $\vec{B}\in \Lambda^q(\mathbb{R}^m)$ and $\vec{C}\in \Lambda^{p-q}(\mathbb{R}^m)$ with $p\geq q$. If $p=1=q$ then we have the usual dot product. Here, we have used the notation $\langle \cdot, \cdot \rangle_{\mathbb{R}^m}$ to emphasise the inner product of multivectors in $\mathbb{R}^m$. We will reserve the notation $\langle\cdot, \cdot\rangle$ for the inner product of forms on patches of the manifold,  this  will be defined and used in Section \ref{chap4}.

\noindent
The {\it bullet dot operator}   $\bullet: \Lambda^p(\mathbb{R}^m)\times \Lambda^q(\mathbb{R}^m)\rightarrow \Lambda^{p+q-2}(\mathbb{R}^m)$ is well known in geometric algebra as the first-order contraction operator. For a $k$-vector $\vec{A}$ and a $1$-vector $\vec{B}$, we define 
$$\vec{A}\bullet \vec{B}=\vec{A}\mathlarger{\mathlarger{\mathlarger{\mathlarger{\llcorner}} }} \vec{B}.$$
For a $p$-vector $\vec{B}$ and a $q$-vector $\vec{C}$, we have
$$\vec{A}\bullet (\vec{B}\wedge\vec{C})=(\vec{A}\bullet \vec{B})\wedge \vec{C}+ (-1)^{pq}(\vec{A}\bullet\vec{C})\wedge\vec{B}.$$

\end{enumerate}

\vskip 3mm

\noindent
{\bf Acknowledgements.} This work forms a significant portion of the second author's PhD Thesis at Monash University. The authors are grateful to Rod Gover, Yannick Sire, Jean Van Schaftingen, Pierre Bousquet and Eduard Curca  for stimualting discussions. The second author is deeply grateful for the Monash Graduate Scholarship he received during his doctoral studies, various support from the School of Mathematics at Monash University and the excellent working environment. The authors would also like to thank the referees for their useful comments.

\section{Variational Derivations and Conservation Laws}\label{latty1}

\noindent
We derive the Willmore-type equation for the  energy $\mathcal{E}(\Sigma)$. This  sixth-order, elliptic, nonlinear equation was first obtained in \cite{robingraham} and \cite{zhang} in arbitrary codimension\footnote{There is a mistake in \cite{robingraham}, in higher codimension, which was corrected in \cite{peterthesis}.}. In order to study this equation \footnote{We will study the Willmore-Euler-Lagrange equation in codimension one. However, our method works in higher codimension at the cost of really cumbersome notations.}, we first establish conservation laws satisfied by the critical points of $\mathcal{E}(\Sigma)$ under translation, dilation and rotation invariance. From this, we write associated   Willmore-type operator as  the divergence of {\it the stress-energy tensor} $\vec{V}$. This is made precise in Theorems  \ref{WE} and \ref{cons-law}.
In Subsection \ref{sekoko}, we obtain further conservation laws by proving the existence of some 2-forms needed to write some useful equations.
The major achievement of Section \ref{latty1} is the reduction of the complicated sixth-order elliptic nonlinear Willmore-Euler-Lagrange equation for $\mathcal{E}(\Sigma)$ to a  manageable system of second order  partial differential equations. Although largely computational, this section prepares the ground for the local analysis of critical points of $\mathcal{E}(\Sigma)$.

\vskip 3mm

\subsection {Willmore-Euler-Lagrange type equation }\label{sekiki}
\begin{df}
Let  $\vec{\Phi}:\Sigma\rightarrow\mathbb{R}^5 $ be an immersion of a 4-dimensional hypersurface $\Sigma$. Let $\mathcal{E}(\Sigma)$ be the Willmore-type energy
$$\mathcal{E}(\Sigma):=\int_{\Sigma}\left(| \nabla H|^2- H^2h^2+ 7H^4\right) \, d\textnormal{vol}_g.$$ 
The manifold $\Sigma$ is Willmore-type for $\mathcal{E}(\Sigma)$ if it is a critical point of $\mathcal{E}(\Sigma).$
\end{df}
\noindent
An important characterisation of Willmore-type immersions can be deduced from the Willmore-Euler-Langrange equation associated with $\mathcal{E}(\Sigma)$ which we now state. 
\begin{theo} \label{WE}
Let  $\vec{\Phi}:\Sigma\rightarrow\mathbb{R}^m $ be an immersion of a 4-dimensional manifold $\Sigma$.  
\noindent
  $\vec{\Phi}$ is Willmore-type for $\mathcal{E}(\Sigma)$ if and only if $\vec{\Phi}$ satisfies the following Willmore-Euler-Lagrange equation  
\begin{align}
\vec{\mathcal{W}}=\vec{0}            \label{abc}
\end{align}
where 
\begin{align}
\vec{\mathcal{W}}&:=-\frac{1}{2}(\Delta^2H)\vec{n}-\frac{1}{2}(|h|^2\Delta H)\vec{n}-4|\nabla H|^2\vec{H} 
+2(\nabla_iH\nabla_jH)\vec{h}^{ij}
\nonumber
                                \\&\quad-\frac{1}{2}\Delta(H|h|^2)\vec{n}-2\nabla_i\nabla_k(H^2 h^{ik})\vec{n} - 2H^2 h^{ik}h_{mk}\vec{h}_i^m-28|H|^4\vec{H}                                                 \nonumber
\\&\quad      -\frac{1}{2}|h|^{4}\vec{H}        \nonumber
   +7 \Delta(|H|^2H)\vec{n}  +11|Hh|^2\vec{H} \nonumber
\end{align}
where $\Delta$ is the negative  Laplace Beltrami operator.

\end{theo}
The proof of Theorem \ref{WE} can be found in \cite[Proposition 5.6]{robingraham} with the mean curvature rescaled, and in \cite[Theorem 2.2.2]{peterthesis} in higher codimension.

\begin{rmk}(see \cite[Page 25]{peterthesis})
At the critical point of $\mathcal{E}(\Sigma)$, there is a scalar tensor $V$ satisfying  
\begin{align}
\delta\int (\vec{B}\cdot\mathcal{\vec{W}}+\nabla_j V^j)\,\,d\textnormal{vol}_g=0  \nonumber
\end{align}
where $\vec{B}$ is the normal part of the variation of $\vec{\Phi}$ and
\begin{align}
V^j&= \left(|\nabla H|^2 -H^2h^2+7H^4\right)A^j      +\frac{1}{2}(\vec{B}\cdot\nabla^j\vec{H})|h|^2  +\frac{1}{2}\nabla^j H \Delta(\vec{B}\cdot\vec{n})     \nonumber
\\&\quad -\frac{1}{2}\nabla^j(\vec{B}\cdot\vec{n})\Delta H+\frac{1}{2}(\vec{B}\cdot\vec{n})\nabla^j\Delta H     -2H^2 (\vec{h}^{ij}\cdot\nabla_i\vec{B})  +2\vec{B}\cdot\nabla_i(H^2\vec{h}^{ij})                 \nonumber
\\& \quad-\frac{1}{2}(\vec{H}\cdot\nabla^j\vec{B} )|h|^2     
  +\frac{1}{2}\vec{B}\cdot\nabla^j(|h|^2\vec{H}) +7 |\vec{H}|^2\vec{H}\cdot\nabla^j\vec{B}-7\vec{B}\cdot\nabla^j(|\vec{H}|^2 \vec{H}).            \nonumber
\end{align}
\end{rmk}

\begin{rmk}
As already pointed out by Zhang \cite{zhang} and Robin-Graham-Reichert \cite{robingraham}, examples of critical points of $\mathcal{E}(\Sigma)$ include  minimal submanifolds, totally geodesic submanifolds $\Sigma$ of $\mathbb{R}^m$, round spheres $\mathbb{S}^4$ of $\mathbb{R}^m$ and products of spheres of radius $r$ such as:  $\mathbb{S}^2(r_1)\times \mathbb{S}^2(r_2)$, $\mathbb{S}^1(r_1)\times \mathbb{S}^3(r_2)$, $\mathbb{S}^1(r_1)\times \mathbb{S}^1(r_2)\times \mathbb{S}^2(r_3)$ and $\mathbb{S}^1(r_1)\times \mathbb{S}^1(r_2)\times \mathbb{S}^1(r_3)\times \mathbb{S}^1(r_4)$  where $\sum r_i^2=1.$
 Indeed, these are all solutions of  the sixth order non-linear elliptic partial differential equation \eqref{abc}.
\end{rmk}

\subsection{The other energies} \label{othergen}

The energy \eqref{jkas} is unbounded from below as observed by Robin-Graham and Reichert in \cite{robingraham} but we can add   the energy $\mu\int_{\Sigma}|h_0|^4\,\, d\textnormal{vol}_g$ to it in order to obtain a nonnegative energy. This holds for $\mu\geq\frac{1}{12}$. In this section, we will investigate these other nonnegative energies in codimension one. 

\begin{prop}[See Proposition 6.6 of \cite{robingraham}]
Let $\vec{\Phi}:\Sigma\rightarrow\mathbb{R}^5$ be an immersion. If $\mu\geq\frac{1}{12}$, then the energy $\mathcal{E}(\Sigma)+\mu\int_\Sigma |h_0|^4\,\, d\textnormal{vol}_g$ is nonnegative.
\end{prop}
\begin{proof}
By using the formula $(h_0)_{ij}=h_{ij}-Hg_{ij}$, we have
\begin{align}
\mathcal{E}(\Sigma)+\mu\int_{\Sigma}|h_0|^4&= \int_{\Sigma} |\nabla H|^2 -|H|^2(|h_0|^2 +4|H|^2) +7|H|^4 +\mu |h_0|^4  \nonumber
\\&= \int_{\Sigma} |\nabla H|^2  +3|H|^4 -|H|^2|h_0|^2 +\mu |h_0|^4 .  \nonumber
\end{align}
Since $\int_{\Sigma} |\nabla H|^2\geq 0$, a sufficient condition for 
$\mathcal{E}(\Sigma)+\mu\int_{\Sigma}|h_0|^4$ to be nonnegative is that the quadratic form $3x^2-xy+\mu y^2$ be nonnegative definite. This gives that $3\mu-\frac{1}{4}\geq 0$ which implies $\mu\geq\frac{1}{12}$.
\end{proof}

\begin{rmk}\label{idf}
Among all energies of the form 
$$\mathcal{E}(\Sigma)+\mu\int_\Sigma |h_0|^4\,\, d\textnormal{vol}_g\,,$$we strongly suspect that the only energy whose critical points include minimal hypersurfaces is $\mathcal{E}(\Sigma)$.
\end{rmk}
We now give an argument to substantiate this remark. Since  minimal hypersurfaces are critical points of $\mathcal{E}(\Sigma)$, we only need to vary  the energy $\int_{\Sigma}|h_0|^4 \,\, d\textnormal{vol}_g$ and show that its critical points do not include minimal hypersurfaces.

\vskip3mm
\noindent
By using the decomposition $(h_0)_{ij}=h_{ij}-Hg_{ij}$ we find that
\begin{align}
|h_0|^4=|h|^4 -8|Hh|^2+16|H|^4 . \nonumber
\end{align}
The variations of $\int_\Sigma|Hh|^2 \,\,d\textnormal{vol}_g$ and $\int_\Sigma|H|^4\,\,d\textnormal{vol}_g$ have already been done in the proof of Theorem \eqref{WE} (see \cite{peterthesis}) and clearly their critical points include minimal hypersurfaces. So, we only need to vary $\int_\Sigma |h|^4\,\,d\textnormal{vol}_g$.

\vskip2mm
\noindent
We find
\begin{align}
\delta |h|^2=\delta(\vec{h}_{ij}\cdot\vec{h}^{ij}) = A^s \nabla_s(\vec{h}_{ij}\cdot\vec{h}^{ij}) + 4(\vec{B}\cdot\vec{h}^{is})(\vec{h}_{ij}\cdot\vec{h}^j_s) + 2\vec{h}^{ij}\cdot\nabla_i\nabla_j\vec{B}                 . \nonumber
\end{align}
\noindent
Hence
\begin{align}
\delta|h|^4 &= 2|\vec{h}|^2\delta|\vec{h}|^2                                    \nonumber
\\& = A^s\nabla_s|h|^4 + 8(\vec{B}\cdot\vec{h}^{is})(\vec{h}_{ij}\cdot\vec{h}^j_s) |h|^2+ 4(\vec{h}^{ij}\cdot\nabla_i\nabla_j\vec{B}) |h|^2                   \nonumber
\end{align}
and
\begin{align}
|g|^{-1/2} \delta(|h|^4 |g|^{1/2})&=       \nabla_s(A^s |h|^4) + 8(\vec{B}\cdot \vec{h}^{is}) (\vec{h}_{ij}\cdot\vec{h}^j_s) |h|^2                                   \nonumber
\\&\quad\quad\quad + 4(\vec{h}^{ij}\cdot\nabla_i\nabla_j\vec{B})|h|^2-4|h|^4\vec{B}\cdot\vec{H} \nonumber
\\&= \nabla_s V_3^s+ \vec{B}\cdot\vec{\mathcal{W}}_3 \nonumber
\end{align}
where
\begin{align}
V_3^s&:= A^s|h|^4 +4(\vec{h}^{sj}\cdot\nabla_j\vec{B})|h|^2 -4\vec{B}\cdot\nabla_j(\vec{h}^{js}|h|^2) \nonumber
\end{align}
and
\begin{align}
\vec{\mathcal{W}}_3 := 4\vec{h}^{is}(\vec{h}_{ij}\cdot\vec{h}^j_s)|h|^2 -4|h|^4\vec{H} +4\nabla_j\nabla_i(h^{ij}|h|^2) \vec{n}      . \label{jagaban78}
\end{align}

The critical point of the energy $\int_\Sigma |h|^4\,\, d\textnormal{vol}_g$ satisfy
\begin{align}
\int_{\Sigma_0}(\vec{B}\cdot\mathcal{\vec{W}}_3+\nabla_jV^j_3)\,\, d\textnormal{vol}_g=0 .\label{kuj11hjk}
\end{align}
As \eqref{kuj11hjk} holds on every small patch $\Sigma_0$ of the manifold $\Sigma$, we can write
$$\vec{B}\cdot\mathcal{\vec{W}}_3+\nabla_jV^j_3=0.$$
The Euler-Lagrange equation for $\mathcal{E}(\Sigma)$ corresponds to the normal variation of $\int_\Sigma |h|^4\,\, d\textnormal{vol}_g$, that is $\pi_{\vec{n}}\delta\vec{\Phi}=\vec{B}$.
This gives
\begin{align}
\mathcal{\vec{W}}_3=\vec{0}.  \label{zihuaguoyann}
\end{align}
We will now argue that minimal hypersurfaces do not satisfy \eqref{zihuaguoyann}.
Assume that $\Sigma$ is a minimal hypersurface. 
The minimality assumption implies that 
\begin{align}
H=\frac{1}{4}h_j^j=0.  \label{minim420}
\end{align}  
By using \eqref{minim420} and the contracted Codazzi-Mainardi equation $(\nabla_i h_{j}^j=\nabla_j h_{i}^j)$ in \eqref{jagaban78}, we find
\begin{align}
\mathcal{\vec{W}}_3&= 4\textnormal{Tr}\vec{h}^3|h|^2  +4\vec{h}^{ij}\nabla_i\nabla_j|h|^2 . \label{hard}
\end{align}
It is hard to imagine that the assumption $H=0$ would make the right hand side of \eqref{hard} vanish. However, we do not rule out the possibility of having a minimal hypersurface for which the right hand side of \eqref{hard} would vanish. Thus, we strongly suspect that for any $\mu$, the critical points of the energy \begin{align}
\mathcal{E}(\Sigma)+\mu\int_\Sigma|h_0|^4\,\,d\textnormal{vol}_g \label{NO00}
\end{align}
 do not include minimal hypersurfaces.

\begin{rmk}
One could also add another curvature term to have
\begin{align}
\mathcal{E}(\Sigma)+\int_\Sigma \sigma\textnormal {Tr}(h_0^4) +\mu|h_0|^4\,\,d\textnormal{vol}_g\quad\quad \sigma,\mu\in \mathbb{R}. \label{NO01}
 \end{align}
However, computations show a strong evidence that minimal hypersurfaces are not critical points of the energies  \eqref{NO00} and \eqref{NO01}.
\end{rmk}

\subsection{Conservation laws}  \label{sekaka}
In this section, the left hand side of the Willmore-type equation \eqref{abc} is reformulated in terms of some divergence free quantities with the help of Noether's theorem.

\vskip3mm
\noindent
Let $\Omega$ be an open subset of $D\subset \mathbb{R}^n$ and let $E\subset \mathbb{R}^m$. Let the Lagrangian
$$L:\big\{ (x,y,z): (x,y)\in D\times E,\, z\in T_y E\otimes T^* _x D \big\}\longrightarrow \mathbb{R}$$
be continuously differentiable. By choosing a $C^1$ density measure $d\mu (x)$ on $\Omega$, we can define a functional $\mathcal{L}$ on the set of maps $C^1(\Omega, E)$ via 
$$\mathcal{L}(u):= \int_{\Omega} L(x, u(x), du(x))\, d\mu(x).$$
A tangent vector $X$ is an {\it infinitesimal symmetry} for $\mathcal{L}$ if and only if 
$$\frac{\partial L}{\partial y^i}(x, y, z) X^i(y) +\frac{\partial L}{\partial z^i_\alpha}(x, y, z) \frac{\partial X^i}{\partial y^j}(y)z^j_\alpha =0.$$
The vector field $J$ defined by  
$$J^\alpha:= \rho(x)\frac{\partial L}{\partial z^i_\alpha}(x, u, du) X^i(u),$$
where $\rho$ satisfies $d\mu(x)= \rho(x) dx^1 \cdots dx^n$, is the well known {\it Noether current} in  Physics. We are now ready to state a version of Noether's theorem relevant for our purpose. 

\begin{theo}\label{thirstyy}
Let $X$ be a Lipschitz tangent vector field on $E$, which is an infinitesimal symmetry for $\mathcal{L}$. If $u:\Omega \rightarrow  E$ is a critical point of $\mathcal{L}$, then 
$$\sum_{\alpha=1}^n \frac{\partial J^\alpha}{\partial x^\alpha}  =0,$$
where $\{ x^\alpha\}_{\alpha=1,...,n}$ are coordinates on $\Omega$ such that $d\mu(x)= \rho(x) dx^1 \cdots dx^n$.
\end{theo}

\noindent
The following result holds by applying Noether's theorem to the invariances of the energy $\mathcal{E}(\Sigma).$
\begin{theo}\label{cons-law}
Let $\vec{\Phi}:\Sigma\rightarrow \mathbb{R}^5$ be a smooth immersion of a 4-dimensional hypersurface $\Sigma$ and let $\mathcal{\vec{W}}$ be as defined in Proposition \ref{WE} with $\mathcal{\vec{W}}=\vec{0}$. Define the following quantity
\begin{align}
\vec{V}^j:&=  |\nabla H|^2\nabla^j\vec{\Phi}-2(\nabla^jH \nabla_iH)\nabla^i\vec{\Phi}+\frac{1}{2}(h^j_i\Delta H)\nabla^i\vec{\Phi}  +\frac{1}{2} (\nabla^j\Delta H)\vec{n} \nonumber
\\&\quad -2\nabla_i(H^2\vec{h}^{ij})+4\nabla_i(H^2h^{ij})\vec{n}     -\frac{1}{2}\nabla^j(|h|^2\vec{H}) -H^2h^2\nabla^j\vec{\Phi}      \nonumber \\&\quad+\nabla^j(|h|^2H)\vec{n}+7|\vec{H}|^4\nabla^j\vec{\Phi}+7\nabla^j(|\vec{H}|^2\vec{H})-14\nabla^j(|\vec{H}|^2H)\vec{n}. \nonumber
\end{align}

\noindent
Then the following conservation laws hold by the invariance of $\mathcal{E}(\Sigma)$ under translation, dilation and rotation respectively.
 \[   \left\{
\begin{array}{ll}
    \vec{0}&=\nabla_j\vec{V}^j\\
      0&=\nabla_j(\vec{\Phi}\cdot\vec{V}^j-\nabla^j|\vec{H}|^2) \\
    \vec{0}&=\nabla_j(\vec{\Phi}\wedge\vec{V}^j+ \vec{M}^j)\\
\end{array} 
\right. \]
where $\vec{M}^j$ denotes the two-vector
\begin{align}\vec{M}^j:&=\frac{1}{2}(\Delta H)\vec{n}\wedge\nabla^j\vec{\Phi}+2(Hh^{ij})\vec{H}\wedge \nabla_i\vec{\Phi}\nonumber
+\frac{1}{2}|h|^2\vec{H}\wedge\nabla^j\vec{\Phi}-7|\vec{H}|^2\vec{H}\wedge\nabla^j\vec{\Phi}.  \nonumber\end{align}

\end{theo}

\begin{proof}

$\empty$\\
\textbf{Translation.}
Consider a translation $\delta\vec{\Phi}=\vec{a}$, where $\vec{a}$ is a constant vector in $\mathbb{R}^5$. Then 
$$A_j=\vec{a}\cdot\nabla_j\vec{\Phi}\quad\mbox{and}\quad \vec{B}=\pi_{\vec{n}}\vec{a}.$$

\noindent
If $\vec{U}$ is any vector, then it holds
$$\vec{U}\cdot\vec{B}=\vec{a}\cdot\pi_{\vec{n}}\vec{U}$$
and
\begin{align} \pi_{\vec{n}}\vec{U}\cdot\nabla_j\vec{B}&=\nabla_j(\vec{B}\cdot\pi_{\vec{n}}\vec{U}) -\vec{B}\cdot \nabla_j\pi_{\vec{n}}\vec{U}   \nonumber
\\&= \vec{a}\cdot \nabla_j\pi_{\vec{n}}\vec{U} -\vec{a}\cdot\pi_{\vec{n}}\nabla_j\pi_{\vec{n}} \vec{U} \nonumber
\\&=\vec{a}\cdot\pi_{T}\nabla_j\pi_{\vec{n}}\vec{U}. \nonumber
\end{align}
With the help of the latter, we have 
\begin{align}
\vec{U}\cdot\Delta_{\perp}\vec{B}&= \pi_{\vec{n}}\vec{U}\cdot\nabla^i\pi_{\vec{n}}\nabla_i\vec{B}=\nabla^i(\pi_{\vec{n}}\vec{U}\cdot\nabla_i\vec{B})-(\pi_{\vec{n}}\nabla^i\pi_{\vec{n}}\vec{U})\cdot \nabla_i\vec{B}   \nonumber
\\&= \vec{a}\cdot \left[ \nabla^i\pi_{T}\nabla_i\pi_{\vec{n}}\vec{U}-\pi_{T}\nabla_i\pi_{\vec{n}}\nabla^i\pi_{\vec{n}}\vec{U}     \right].    \nonumber
\end{align}
Hence
$$V^j=\vec{a}\cdot\vec{V}^j,$$
where
\begin{align}
\vec{V}^j&= | \nabla H  |^2\nabla^j\vec{\Phi}+\frac{1}{2} (|h|^2 \nabla^j H)\vec{n}+ \frac{1}{2}(\nabla^j\Delta H)\vec{n}-\frac{1}{2}\pi_{T}\nabla^j(\Delta H\vec{n}) \nonumber
\\ &\quad+\frac{1}{2} \nabla^i\pi_{T}\nabla_i(\nabla^j H\vec{n})-\frac{1}{2}\pi_{T}\nabla_i\pi_{\vec{n}}\nabla^i\pi_{\vec{n}}\nabla^j\vec{H}  -2\pi_{T}\nabla_i(H^2\vec{h}^{ij}) \nonumber
\\&\quad 
+ 2\pi_{\vec{n}} \nabla_i(H^2\vec{h}^{ij})-\frac{1}{2}\pi_{T}\nabla^j(|h|^2\vec{H}) \nonumber
+\frac{1}{2}\pi_{\vec{n}}\nabla^j(|h|^2\vec{H})-H^2h^2 \nabla^j\vec{\Phi}              \nonumber
\\&\quad+7|\vec{H}|^2\nabla^j\vec{\Phi}+7\pi_{T}\nabla^j(|\vec{H}|^2\vec{H}) -7\pi_{\vec{n}}\nabla^j(|\vec{H}|^2\vec{H})
 \nonumber
\\&=|\nabla H|^2\nabla^j\vec{\Phi}-2(\nabla^jH \nabla_iH)\nabla^i\vec{\Phi}+\frac{1}{2}(h^j_i\Delta H)\nabla^i\vec{\Phi}  +\frac{1}{2} (\nabla^j\Delta H)\vec{n} \nonumber
\\&\quad -2\nabla_i(H^2\vec{h}^{ij})+4\nabla_i(H^2h^{ij})\vec{n}     -\frac{1}{2}\nabla^j(|h|^2\vec{H}) -H^2h^2\nabla^j\vec{\Phi}      \nonumber \\&\quad+\nabla^j(|h|^2H)\vec{n}+7|\vec{H}|^4\nabla^j\vec{\Phi}+7\nabla^j(|\vec{H}|^2\vec{H})-14\nabla^j(|\vec{H}|^2H)\vec{n}. \nonumber
\end{align}

\noindent
We note that 
$$\nabla_j\vec{\Phi}\cdot \vec{V}^j=\Delta|\vec{H}|^2.$$

\vskip 3mm
\textbf{Dilation.}
Next, we consider a dilation $\delta\vec{\Phi}=\lambda\vec{\Phi}$,  for some constant $\lambda\in\mathbb{R}$. Clearly,
$$A_j=\lambda\vec{\Phi}\cdot\nabla_j\vec{\Phi}\quad\mbox{and}\quad\vec{B}=\lambda\pi_{\vec{n}}\vec{\Phi}.$$

\noindent
If $\vec{U}$ is any vector, then it holds
$$\vec{U}\cdot\vec{B}=\lambda\vec{\Phi}\cdot\pi_{\vec{n}}\vec{U}$$
and
\begin{align}
\quad\pi_{\vec{n}}\vec{U}\cdot\nabla_j\vec{B}&=  \nabla_j(\vec{B}\cdot\pi_{\vec{n}}\vec{U}) -\vec{B}\cdot \nabla_j\pi_{\vec{n}}\vec{U}             \nonumber
\\&= \lambda\left(\vec{\Phi}\cdot \nabla_j\pi_{\vec{n}}\vec{U}-\pi_{\vec{n}}\nabla_j\pi_{\vec{n}}\vec{U}  \right)   \nonumber
\\&=\lambda \vec{\Phi}\cdot\pi_{T}\nabla_j\pi_{\vec{n}}\vec{U}.  \nonumber
\end{align}
Hence
\begin{align}
\vec{U}\cdot \Delta_{\perp}\vec{B}&= \pi_{\vec{n}}\vec{U}\cdot\nabla^i\pi_{\vec{n}}\nabla_i\vec{B}=\nabla^i(\pi_{\vec{n}}\vec{U}\cdot \nabla_{i}\vec{B})-(\pi_{\vec{n}}\nabla^i\pi_{\vec{n}}\vec{U})\cdot \nabla_i\vec{B}   \nonumber
\\&= \nabla^i\left[\lambda \vec{\Phi}\cdot \pi_{T}\nabla_i\pi_{\vec{n}}\vec{U}  \right]-\lambda\vec{\Phi}\cdot \pi_{T}\nabla_{i}\pi_{\vec{n}}\nabla^i\pi_{\vec{n}}\vec{U}  \nonumber
\\&= \lambda\vec{\Phi}\cdot \left[ \nabla^i\pi_{T}\nabla_i\pi_{\vec{n}}\vec{U}-\pi_{T}\nabla_i\pi_{\vec{n}}\nabla^i\pi_{\vec{n}}\vec{U}  \right]-4\lambda\vec{H}\cdot \vec{U}   \nonumber
\end{align}
 From this and the computation above for  translation, we find
$$V^j=\lambda\vec{\Phi}\cdot\vec{V}^j-2\lambda\vec{H}\cdot\nabla^j\vec{H}=\lambda\left[ \vec{\Phi}\cdot\vec{V}^j-\nabla^j|\vec{H}|^2  \right].$$

\vskip 3mm
\noindent\textbf{Rotation.} We now consider an infinitesimal rotation given by\footnote{Here, we use $\star$ to denote the Hodge star operator on the ambient space. } $$\delta\vec{\Phi}=\star(\vec{b}\wedge\vec{\Phi}) \quad\mbox{for some constant}\quad \vec{b}\in\Lambda^{3}(\mathbb{R}^5).$$
Clearly,
$$A_j=\vec{b}\cdot\star(\vec{\Phi}\wedge\nabla_j\vec{\Phi})\quad\mbox{and}\quad \vec{B}=(\vec{b}\cdot\star(\vec{\Phi}\wedge\vec{n}))\vec{n}.$$
If $\vec{U}$ is any vector, then it holds that
$$\vec{U}\cdot\vec{B}=\vec{b}\cdot\star(\vec{\Phi}\wedge\pi_{\vec{n}}\vec{U})$$
and
\begin{align}
\pi_{\vec{n}}\vec{U} \cdot \nabla_j\vec{B}&=\nabla_j(\vec{B}\cdot\pi_{\vec{n}}\vec{U}) -\vec{B}\cdot \nabla_j\pi_{\vec{n}} \vec{U}   \nonumber
\\ &=\vec{b}\cdot\star(\vec{\Phi}\wedge\pi_{T}\nabla_j\pi_{\vec{n}}\vec{U}+\nabla_j\vec{\Phi}\wedge\pi_{\vec{n}}\vec{U}). \nonumber\end{align}

\noindent
Hence 
\begin{align}
\vec{U}\cdot\Delta_{\perp}\vec{B}&=\pi_{\vec{n}}\vec{U}\cdot\nabla^i\pi_{\vec{n}}\nabla_i\vec{B}=\nabla^i(\pi_{\vec{n}}\vec{U} \cdot \nabla_i\vec{B})-(\pi_{\vec{n}}\nabla^i\pi_{\vec{n}}\vec{U} )\cdot \nabla_i\vec{B}    \nonumber
\\&=  \nabla^i\left(\vec{b}\cdot\star (\vec{\Phi}\wedge\pi_{T}\nabla_i\pi_{\vec{n}}\vec{U}+\nabla_i\vec{\Phi}\wedge\pi_{\vec{n}} \vec{U})  \right)         \nonumber \\&\quad\quad-\vec{b}\cdot\star\left(\vec{\Phi}\wedge\pi_{T}\nabla_i\pi_{\vec{n}}\nabla^i\pi_{\vec{n}}\vec{U}+\nabla_i\vec{\Phi}\wedge\pi_{\vec{n}}\nabla^i\pi_{\vec{n}}\vec{U}\right)   \nonumber
\\&= \vec{b}\cdot\star \left[\vec{\Phi}\wedge\left(\nabla^i\pi_{T}\nabla_i\pi_{\vec{n}}\vec{U}-\pi_{T}\nabla_i\pi_{\vec{n}} \nabla^i\pi_{\vec{n}}\vec{U} \right) +\nabla^i\vec{\Phi}\wedge\pi_{T}\nabla_i\pi_{\vec{n}}\vec{U}                    \right.\nonumber
\\&\quad\quad\left. +\nabla^i(\nabla_i\vec{\Phi}\wedge\pi_{\vec{n}}\vec{U})
-\nabla_i\vec{\Phi}\wedge\pi_{\vec{n}}\nabla^i\pi_{\vec{n}}\vec{U} \right]  \nonumber
\\&=\vec{b}\cdot\star\left[\vec{\Phi}\wedge\left( \nabla^i\pi_{T}\nabla_i\pi_{\vec{n}}\vec{U}-\pi_{T}\nabla_i\pi_{\vec{n}} \nabla^i\pi_{\vec{n}}\vec{U} \right) +2\nabla^i\vec{\Phi}\wedge \pi_{T}\nabla_i\pi_{\vec{n}}\vec{U} \right]   \nonumber
\\&=\vec{b}\cdot\star\left[ \vec{\Phi}\wedge\left( \nabla^i\pi_{T}\nabla_i\pi_{\vec{n}}\vec{U}-\pi_{T}\nabla_i\pi_{\vec{n}}\nabla^i\pi_{\vec{n}}\vec{U} \right) \right].   \nonumber
\end{align}
From this and the computation above for translation, we find
\begin{align}
V^j&=\vec{b}\cdot\star\left[ \vec{\Phi}\wedge\vec{V}^j+ \frac{1}{2}(\Delta H)\vec{n}\wedge\nabla^j\vec{\Phi}       +2H^2\vec{h}^{ij}\wedge\nabla_i\vec{\Phi} \right. \nonumber
\\&\quad\quad\left.+\frac{1}{2}|h|^2\vec{H}\wedge\nabla^j\vec{\Phi} -7|\vec{H}|^2\vec{H}\wedge\nabla^j\vec{\Phi}\right] . \nonumber
\end{align}
\end{proof}

\subsection{Existence of potential 2-forms} \label{sekoko}

In this subsection, we present further conservation laws. We briefly clarify some of the notations used here.
Most of the quantities used are multivector-valued differential forms. The space $\Lambda^k(\mathbb{R}^4, \Lambda^p(\mathbb{R}^5))$ is the collection of all $k$-forms that act as $p$-vectors from $\mathbb{R}^4$ to $\mathbb{R}^5$. Thus an element of this space is both a $p$-vector and a $k$-form, or simply a $p$-vector-valued $k$-form. For $\vec{A}\in\Lambda^k(\mathbb{R}^4, \Lambda^p(\mathbb{R}^5))$ and $\vec{B}\in\Lambda^\ell(\mathbb{R}^4, \Lambda^q(\mathbb{R}^5))$, the product $\vec{A}\overset{\bullet}\wedge_4 \vec{B} \in \Lambda^{k+\ell}(\mathbb{R}^4, \Lambda^{p+q-2}(\mathbb{R}^5))$ represents the contraction between multivectors and wedge product between forms. Similar meaning is given to the products $\vec{A}\overset{\cdot}\wedge_4 \vec{B}$ and $\vec{A}\overset{\wedge}\wedge_4 \vec{B}$ where the upper and lower symbols denote operations between vectors and forms respectively. We will simply write $\vec{A}\wedge \vec{B}$ whenever the meaning is clear. For instance, when $\vec{A}$ is a 1-vector-valued 0-form and $\vec{B}$ is a $q$-vector valued $\ell$-form.

We recall the following useful result due to Henri Poincar\'e (see \cite{jost}).

\noindent
\begin{lem}[Poincar\'e lemma]
Let $\Omega\subset \mathbb{R}^4$ be an open ball. Let $\omega\in \Lambda^k(\Omega, \Lambda^p(\mathbb{R}^5))$ with $1\leq k\leq 4$ satisfy
$$d\omega=0.$$
Then $\omega=d\alpha$ for some $\alpha\in \Lambda^{k-1}(\Omega, \Lambda^p(\mathbb{R}^5))$. Similarly, if $0\leq k\leq 3$ and
$$d^\star \omega=0,$$
then $\omega=d^\star\alpha$ for some $\alpha\in \Lambda^{k+1}(\Omega, \Lambda^p(\mathbb{R}^5))$.
\end{lem}

\begin{prop}\label{coro}
Let $\vec{\Phi}:\Sigma \rightarrow \mathbb{R}^5$ be a critical point of $\mathcal{E}(\Sigma)$ and let $\vec{V}$ be as defined in Theorem \ref{cons-law}.
Then it holds that
\begin{gather}
d^{\star} \vec{V}=\vec{0}, \nonumber
\\ d^\star (\vec{\Phi}\cdot\vec{V}-d|\vec{H}|^2)=0  \nonumber
 \\ \mbox{and}  \quad d^\star(\vec{\Phi}\wedge\vec{V} +(\vec{J}+2d^\star(|\vec{H}|^2d\vec{\Phi}))\wedge d\vec{\Phi})=\vec{0}  \nonumber
\end{gather}
where $\vec{J}:=\frac{1}{2}(\Delta H)\vec{n}+\frac{1}{2}\vec{H}|\vec{h}|^{2}-7|\vec{H}|^2\vec{H}$.
Also, there exist two-forms $\vec{L}_0,\vec{L}\in\Lambda^2(\mathbb{R}^4, \Lambda^1(\mathbb{R}^5))$, $S_0\in\Lambda^2(\mathbb{R}^4, \Lambda^0(\mathbb{R}^5))$ and $\vec{R}_0\in\Lambda^2(\mathbb{R}^4, \Lambda^2(\mathbb{R}^5))$ such that
\begin{gather}
\vec{V}= d^\star\vec{L}_0 \nonumber\\
dS_0= \vec{L}\overset{\cdot}\wedge_4 d\vec{\Phi} \nonumber\\
d\vec{R}_0=\vec{L}\overset{\wedge}\wedge_4 d\vec{\Phi} +\star d\vec{u} -(\vec{J} +8 |\vec{H}|^2\vec{H})\wedge \star d\vec{\Phi} \nonumber
\end{gather}
where $\vec{L}_0$ and $\vec{L}$ are related by $\vec{L}=\star(\vec{L}_0-d|\vec{H}|^2\wedge_4 d\vec{\Phi})$ and $\vec{u}$ satisfies the Hodge decomposition
\begin{gather}
\frac{5}{3}|\vec{H}|^2d^\star \vec{\eta}= d\vec{u}+d^\star \vec{v},\nonumber\\
 \quad \vec{u}\in\Lambda^0(\mathbb{R}^4, \Lambda^2(\mathbb{R}^5)), \vec{v}\in\Lambda^2(\mathbb{R}^4, \Lambda^2(\mathbb{R}^5)), \vec{\eta}:=d\vec{\Phi}\overset{\wedge}\wedge_4 d\vec{\Phi}\in\Lambda^2(\mathbb{R}^4, \Lambda^2(\mathbb{R}^5)).  \nonumber
\end{gather}
\end{prop}

\begin{proof}
From Theorem \ref{cons-law}, we have $\nabla_j\vec{V}^j=\vec{0}.$ In other words, $d^\star\vec{V}=\vec{0}$ where $\vec{V}\in \Lambda^1(\mathbb{R}^4, \Lambda^1(\mathbb{R}^5))$. By the Poincar\'e lemma, there exists some $\vec{L}_0\in\Lambda^2(\mathbb{R}^4, \Lambda^1(\mathbb{R}^5))$ such that 
\begin{gather}
\vec{V}=d^\star \vec{L}_0          \label{lqw}
\end{gather}
with the choice
$$d\vec{L}_0=\vec{0}.$$
Also, from Theorem \ref{cons-law}, the invariance of $\mathcal{E}(\Sigma)$ under dilation implies that  
\begin{gather}
d^\star(\vec{\Phi}\cdot\vec{V}-d|\vec{H}|^2)=\nabla_j(\vec{V}^j\cdot\vec{\Phi}-\nabla^j|\vec{H}|^2)=0.        \label{lqwer}
\end{gather}
Note that
\begin{align}
\vec{V}\cdot\vec{\Phi}&= d^\star\vec{L}_0\cdot\vec{\Phi}= \star d\star \vec{L}_0\cdot\vec{\Phi}= \star\left( d(\star\vec{L}_0\cdot\vec{\Phi})-(\star\vec{L}_0)\overset{\cdot}\wedge_4 d\vec{\Phi}  \right)  \nonumber
\\&=d^\star(\vec{L}_0\cdot\vec{\Phi})  -\star((\star \vec{L}_0)\overset{\cdot}\wedge_4 d\vec{\Phi})   \nonumber
\end{align}
so that \eqref{lqwer} becomes
\begin{gather}
d^\star\left[ \star((\star \vec{L}_0)\overset{\cdot}\wedge_4 d\vec{\Phi})  +d|\vec{H}|^2  \right]=0   \nonumber
\end{gather}
or equivalently
\begin{gather}
d\left[(\star \vec{L}_0)\overset{\cdot}\wedge_4 d\vec{\Phi} -\star d|\vec{H}|^2     \right]=0.          \nonumber
\end{gather}

\noindent
Hence, there exists $S_0\in \Lambda^2(\mathbb{R}^4, \Lambda^0(\mathbb{R}^5))$ such that
\begin{gather}
dS_0=(\star \vec{L}_0)\overset{\cdot}\wedge_4 d\vec{\Phi} -\star d|\vec{H}|^2.  \nonumber
\end{gather}

\noindent
Further computations show that
\begin{align}
dS_0&=\frac{1}{6} (\star\vec{L}_0)_{[kl}\cdot \nabla_{m]}\vec{\Phi} \,\,\,dx^{k}\wedge_4 dx^l \wedge_4 dx^m -\frac{1}{6}\epsilon_{iklm}\nabla^i|\vec{H}|^2 \,\,\,dx^{k}\wedge_4 dx^l \wedge_4 dx^m   \nonumber
\\&=\frac{1}{6}\epsilon_{ij[kl}\nabla_{m]}\vec{\Phi}\cdot\left( \frac{1}{2}\vec{L}_0^{ij}-\frac{1}{2}\nabla^{[i}|\vec{H}|^2\nabla^{j]}\vec{\Phi} \right)  \,\,\,dx^{k}\wedge_4 dx^l \wedge_4 dx^m  \nonumber
\\&=\star(\vec{L}_0-d|\vec{H}|^2\wedge_4 d\vec{\Phi})\overset{\cdot}\wedge_4 d\vec{\Phi} .   \nonumber
\end{align}
By setting 
$$\vec{L}:=\star(\vec{L}_0-d|\vec{H}|^2\wedge_4 d\vec{\Phi})$$
we find that
\begin{gather}
dS_0=\vec{L}\overset{\cdot}\wedge_4d\vec{\Phi}.  \nonumber
\end{gather}
Note that $d^\star \vec{L}=\star d\vec{L}_0=\vec{0}.$

\noindent
Now, using $\nabla_j\vec{V}^j=\vec{0}$ we have
\begin{align}
\vec{V}^j\wedge\nabla_j\vec{\Phi}&=\nabla_j(\vec{J}\wedge\nabla^j\vec{\Phi} +2|\vec{H}|^2\vec{h}^{ij}\wedge\nabla_i\vec{\Phi})   \nonumber\\
&= \nabla_j(\vec{J}\wedge \nabla^j\vec{\Phi} +2\nabla^i(|\vec{H}|^2\nabla_i\vec{\Phi})\wedge\nabla^j\vec{\Phi})  \nonumber
\end{align}
so that
\begin{gather}
\nabla_j\left( \vec{V}^j\wedge\vec{\Phi}-(\vec{J}+2\nabla^i(|\vec{H}|^2\nabla_i\vec{\Phi})\wedge\nabla^j\vec{\Phi} )\right)=\vec{0}  \nonumber
\end{gather}
or equivalently
\begin{align}
d^\star\left( \vec{V}\wedge\vec{\Phi}-\vec{J}\wedge d\vec{\Phi} -2d^\star(|\vec{H}|^2d\vec{\Phi})\wedge d\vec{\Phi}  \right)=\vec{0} . \label{becc}
\end{align}
Observe that
\begin{align}
\vec{V}\wedge \vec{\Phi}&= d^\star\vec{L}_0\wedge \vec{\Phi}=\star d\star \vec{L}_0\wedge \vec{\Phi}    \nonumber
=       \star\left[d(\star \vec{L}_0\wedge \vec{\Phi})-(\star\vec{L}_0)\overset{\wedge}\wedge_4 d\vec{\Phi}  \right]    \nonumber
\\&=d^\star(\vec{L}_0\wedge\vec{\Phi})-\star\left[(\star\vec{L}_0)\overset{\wedge}\wedge_4 d\vec{\Phi}\right]    \nonumber
\\&=d^\star(\vec{L}_0\wedge\vec{\Phi})- \star(\vec{L}\overset{\wedge}\wedge_4 d\vec{\Phi})  -\star\left[\left(\star(d|\vec{H}|^2\wedge_4 d\vec{\Phi})  \right) \overset{\wedge}\wedge_4 d\vec{\Phi} \right]  .    \label{gracee}
\end{align}

\noindent
We will define $\vec{\eta}\in\Lambda^2(\mathbb{R}^4, \Lambda^2(\mathbb{R}^5))$ as follows
\begin{align}
\vec{\eta}=\frac{1}{2}\vec{\eta}_{ij} \,\,\,dx^{i}\wedge_4 dx^j:= \frac{1}{2} \nabla_i\vec{\Phi}\wedge\nabla_j\vec{\Phi}\,\,\, dx^{i}\wedge_4 dx^j.
\end{align}

\noindent
Focusing on \eqref{gracee}, we see that
\begin{align}
\star\left[\left(\star(d|\vec{H}|^2\wedge_4 d\vec{\Phi})  \right) \overset{\wedge}\wedge_4 d\vec{\Phi} \right]  &= \frac{1}{6}\epsilon^{klmr}\left( \frac{1}{6}\delta^{\alpha\beta\gamma}_{klm}\epsilon_{ij\alpha\beta}\frac{1}{2} \nabla^{[i}|\vec{H}|^2\nabla^{j]}\vec{\Phi} \right)\wedge \nabla_\gamma\vec{\Phi} \,\,g_{rs}\,\,\, dx^s  \nonumber
\\&= \frac{1}{6}\delta_{ij}^{mr} \nabla^{[i}|\vec{H}|^2\nabla^{j]}\vec{\Phi}\wedge \nabla_m\vec{\Phi}  g_{rs} dx^s     \nonumber
  \\&=-\frac{1}{3} \nabla_m|\vec{H}|^2 \nabla^m\vec{\Phi}\wedge\nabla_r\vec{\Phi} dx^r  \nonumber
\\&=-\frac{1}{3}\nabla_m(|\vec{H}|^2\nabla^m\vec{\Phi}\wedge \nabla_r\vec{\Phi}) dx^r +\frac{1}{3} |\vec{H}|^2 \nabla_m(\nabla^m\vec{\Phi}\wedge\nabla_r\vec{\Phi}) dx^r  \nonumber
\\&=-\frac{1}{3} d^\star(|\vec{H}|^2\vec{\eta})+\frac{1}{3}|\vec{H}|^2 d^\star\vec{\eta}  .  \nonumber
\end{align}
Thus, it follows that
\begin{align}
\vec{V}\wedge\vec{\Phi}=  d^\star(\vec{L}_0\wedge\vec{\Phi})- \star(\vec{L}\overset{\wedge}\wedge_4 d\vec{\Phi})     +\frac{1}{3} d^\star(|\vec{H}|^2\vec{\eta})-\frac{1}{3}|\vec{H}|^2 d^\star\vec{\eta}   \nonumber
\end{align}
and so \eqref{becc} becomes
\begin{align}
d^\star\left[-\star(\vec{L}\overset{\wedge}\wedge_4 d\vec{\Phi})  -\frac{1}{3}|\vec{H}|^2 d^\star\vec{\eta} -\vec{J}\wedge d\vec{\Phi} -2d^\star(|\vec{H}|^2d\vec{\Phi})\wedge d\vec{\Phi}      \right]=\vec{0}.   \label{sct}
\end{align}

\noindent
Observe that
\begin{align}
d^\star (|\vec{H}|^2 d\vec{\Phi})\wedge d\vec{\Phi}&= \nabla_j(|\vec{H}|^2\nabla^j\vec{\Phi})\wedge\nabla_i\vec{\Phi}  dx^i       \nonumber
\\&= 4|\vec{H}|^2\vec{H}\wedge\nabla_i\vec{\Phi} dx^i +\nabla_j|\vec{H}|^2\nabla^j\vec{\Phi}\wedge \nabla_i\vec{\Phi}  dx^i  \nonumber
\\&= 4|\vec{H}|^2\vec{H}\wedge\nabla_i\vec{\Phi} dx^i  +\nabla^j(|\vec{H}|^2\nabla_j\vec{\Phi}\wedge\nabla_i\vec{\Phi}) dx^i \nonumber
\\&\quad\quad-|\vec{H}|^2\nabla^j(\nabla_j\vec{\Phi}\wedge\nabla_i\vec{\Phi}) dx^i   \nonumber
\\&= 4|\vec{H}|^2\vec{H}\wedge d\vec{\Phi} +d^\star(|\vec{H}|^2\vec{\eta})  -|\vec{H}|^2d^\star\vec{\eta}   \label{scv}
\end{align}

\noindent
Now, substitute \eqref{scv} into \eqref{sct} to find
\begin{align}
d^\star\left[-\star(\vec{L}\overset{\wedge}\wedge_4 d\vec{\Phi})  +\frac{5}{3}|\vec{H}|^2 d^\star\vec{\eta} -(\vec{J} +8 |\vec{H}|^2\vec{H})\wedge d\vec{\Phi}     \right]=\vec{0}  .\nonumber
\end{align}

\noindent
We introduce the Hodge decomposition 
\begin{align}
\frac{5}{3} |\vec{H}|^2d^\star\vec{\eta}= d\vec{u} +d^\star \vec{v}  \nonumber
\end{align}
so that 
\begin{align}
d^\star\left[ -\star(\vec{L}\overset{\wedge}\wedge_4 d\vec{\Phi})  +d\vec{u}  -(\vec{J} +8 |\vec{H}|^2\vec{H})\wedge d\vec{\Phi}    \right] =\vec{0}   \nonumber
\end{align}
or equivalently
\begin{align}
d\left[\vec{L}\overset{\wedge}\wedge_4 d\vec{\Phi} +\star d\vec{u} -(\vec{J} +8 |\vec{H}|^2\vec{H})\wedge \star d\vec{\Phi}      \right]=\vec{0}.
\end{align}
Therefore there exists $\vec{R}_0\in\Lambda^2(\mathbb{R}^4, \Lambda^2(\mathbb{R}^5))$ such that
\begin{align}
d\vec{R}_0=\vec{L}\overset{\wedge}\wedge_4 d\vec{\Phi} +\star d\vec{u} -(\vec{J} +8 |\vec{H}|^2\vec{H})\wedge \star d\vec{\Phi}  .   \nonumber
\end{align}

\noindent
This completes the proof.

\end{proof}

\noindent
We now state an identity which holds on any 4-dimensional manifold $\Sigma$. This identity is obtained by varying 
\begin{align}
\mathcal{E}_{H^2}:= \int_{\Sigma_0\subset \Sigma}|\vec{H}|^2 d\textnormal{vol}_g.   \label{2DD}
\end{align}
on every small patch $\Sigma_0$ of $\Sigma$ and using the translation invariance of \eqref{2DD}. Note that we are not considering the critical points of the energy \eqref{2DD}. Hence the Willmore-type operator $\vec{\mathcal{W}}_{H^2}$ for \eqref{2DD} does not satisfy 
$$\vec{\mathcal{W}}_{H^2}=\vec{0}.$$

\begin{lem}  \cite{peterthesis}\label{rivvy}
Let $\vec{\Phi}:\Sigma\rightarrow \mathbb{R}^m$ be a smooth immersion of a 4-dimensional manifold $\Sigma$. Define 
\begin{align}
\vec{A}^s:= \nabla^s\vec{H}-2(|\vec{H}|^2g^{sk}-\vec{H}\cdot\vec{h}^{sk})\nabla_k\vec{\Phi}. \label{defA}
\end{align}
Then the following identity holds
\begin{align}
\nabla_s \vec{A}^s= \Delta_\perp\vec{H} +(\vec{H}\cdot \vec{h}_{ij})\vec{h}^{ij} -8|\vec{H}|^2\vec{H}.  \nonumber
\end{align}
\end{lem}

\begin{rmk}
We have proved in Theorem \ref{coro} the existence of 2-forms $\vec{L}$, $S_0$ and $\vec{R}_0$. Observe that $S_0$ and $\vec{R}_0$ are defined in terms of $\vec{L}$ which depends on some geometric quantity $\vec{V}$. The following Lemma shows that $S_0$ and $\vec{R}_0$ can be directly linked back to some geometric quantities via the ``return equation".
\end{rmk}

\begin{lem}[The return equation]
The quantities $S_0$, $\vec{R}_0$, $\vec{u}$, $\vec{v}$ and $\vec{\eta}$  defined in Proposition \ref{coro} satisfy
\begin{gather}
\star (d\vec{R}_0\overset{\bullet}\wedge_4 d\vec{\Phi}) +\star(dS_0 \wedge_4 d\vec{\Phi}) +6 \star(\star d^\star \vec{v}\overset{\bullet}\wedge_4 d\vec{\Phi}) = 12d^\star\left[d\vec{H} -2|\vec{H}|^2d\vec{\Phi} -2\pi_{T} d\vec{H} \right] \nonumber
\\
\Delta_g \vec{u} =\frac{5}{3}\star(d|\vec{H}|^2\wedge_4 d(\star\vec{\eta})) \quad\quad \textnormal{and}\quad\quad \Delta\vec{v}=\frac{5}{3} d(|\vec{H}|^2\wedge_4 d^\star\vec{\eta})\nonumber
\end{gather}
where $\Delta_g$ is the negative Laplace-Beltrami operator and $\Delta$ is the Hodge Laplacian.
\end{lem}
\begin{proof}
We have
\begin{align}
&\quad\quad\epsilon^{mijk}\nabla_i(\vec{R}_0)_{jk}\bullet \nabla_m\vec{\Phi}  \nonumber
\\&= \epsilon^{mijk}\left[   \vec{L}_{jk}\wedge \nabla_i\vec{\Phi}+\epsilon_{rijk}\nabla^r\vec{u}-\epsilon_{rijk}(\vec{J}+8|\vec{H}|^2\vec{H})\wedge\nabla^r\vec{\Phi}                  \right]\bullet \nabla_m\vec{\Phi}    \nonumber
\\&=  \epsilon^{mijk}(\vec{L}_{jk}\cdot\nabla_m\vec{\Phi}) \nabla_i\vec{\Phi} +6\delta^m_r \nabla^r\vec{u}\bullet \nabla_m\vec{\Phi}-6\delta^m_r(\vec{J}+8|\vec{H}|^2\vec{H})(-\delta^r_m)         \nonumber
\\&= \epsilon^{mijk}\nabla_m (S_0)_{jk}\nabla_i\vec{\Phi}+6\nabla_i\vec{u}\bullet\nabla^i\vec{\Phi}+24(\vec{J}+8|\vec{H}|^2\vec{H}).  \nonumber
\end{align}

We have found
\begin{align}
\epsilon^{mijk}(\nabla_i(\vec{R}_0)_{jk}\bullet \nabla_m\vec{\Phi}+ \nabla_i (S_0)_{jk}\nabla_m\vec{\Phi})-6\nabla_i\vec{u}\bullet\nabla^i\vec{\Phi}&=  24(\vec{J}+8|\vec{H}|^2\vec{H}) \label{stella}
\end{align}

\noindent
Equation \eqref{stella} can be further simplified in order to re-introduce $v$.
\noindent
Observe that
\begin{align}
\nabla_i\vec{u}\bullet \nabla^i\vec{\Phi}&= \left[\frac{5}{3} |\vec{H}|^2\nabla_j(\nabla^j\vec{\Phi}\wedge \nabla_i\vec{\Phi}) -\nabla^j\vec{v}_{ji}    \right]\bullet\nabla^i\vec{\Phi}         \nonumber
\\&= \frac{5}{3}|\vec{H}|^2\left( 4\vec{H}\wedge\nabla_i\vec{\Phi}+\nabla^j\vec{\Phi}\wedge\vec{h}_{ji}  \right)\bullet\nabla^i\vec{\Phi}    -\nabla_j\vec{v}^{ji}\bullet\nabla_i\vec{\Phi}   \nonumber
\\&= -20|\vec{H}|^2\vec{H} -\nabla_j\vec{v}^{ji}\bullet\nabla_i\vec{\Phi} . \nonumber
\end{align}
Thus we find that
\begin{align}
\epsilon^{mijk}\nabla_i(\vec{R}_0)_{jk}\bullet \nabla_m\vec{\Phi}+ \epsilon^{mijk} \nabla_i (S_0)_{jk} \nabla_m\vec{\Phi}   +6\nabla_j\vec{v}^{ji}\bullet\nabla_i\vec{\Phi} = 24(\vec{J}+8|\vec{H}|^2\vec{H}) -120|\vec{H}|^2\vec{H}.   \nonumber
\end{align}
Since $\vec{J}:=\frac{1}{2}\Delta_{\perp}\vec{H}+\frac{1}{2}(\vec{H}\cdot\vec{h}^{ij})\vec{h}_{ij} +|\vec{H}|^2\vec{H} $, we have 
\begin{align}
24(\vec{J}+8|\vec{H}|^2\vec{H}) -120|\vec{H}|^2\vec{H}&\nonumber
 =12\Delta_{\perp}\vec{H}+12(\vec{H}\cdot\vec{h}^{ij})\vec{h}_{ij} -96|\vec{H}|^2\vec{H}   \nonumber
\\&= 12\nabla_j\left[\nabla^j\vec{H}-2|\vec{H}|^2\nabla^j\vec{\Phi} +2(\vec{H}\cdot\vec{h}^{jk})\nabla_k\vec{\Phi}  \right]   \label{zea}
\end{align}
where we have used Lemma \ref{rivvy} to write \eqref{zea}.
Hence, we have found
\begin{gather}
\epsilon^{mijk}\nabla_i(\vec{R}_0)_{jk}\bullet \nabla_m\vec{\Phi}+ \epsilon^{mijk} \nabla_i (S_0)_{jk} \nabla_m\vec{\Phi}   +6\nabla_j\vec{v}^{ji}\bullet\nabla_i\vec{\Phi}   \nonumber
\\=12\nabla_j\left[\nabla^j\vec{H}-2|\vec{H}|^2\nabla^j\vec{\Phi} +2(\vec{H}\cdot\vec{h}^{jk})\nabla_k\vec{\Phi}  \right] .   \label{recca}
\end{gather}

\noindent
Lastly, applying divergence to the Hodge decomposition
$$\nabla^j\vec{u}+\nabla_i\vec{v}^{ij}=\frac{5}{3}|\vec{H}|^2\nabla_i\vec{\eta}^{ij}$$
yields
\begin{align}\Delta_g\vec{u}=\frac{5}{3}\nabla_j|\vec{H}|^2\nabla_i\vec{\eta}^{ij}. \label{hads8}\end{align}

\noindent
Also, $\vec{v}$ satisfies the equation
$$\Delta\vec{v}=\frac{5}{3} d\left( |\vec{H}|^2d^\star\vec{\eta}  \right).$$
\end{proof}

\section{Regularity Results} \label{chap4}

\noindent
The proof of Theorem \ref{heart} is divided into different subsections of this section. In Subsection \ref{adzsq}, we develop generic regularity results applicable to our problem. 
The goal of Subsection \ref{amena} is to understand the regularity of the quantities $\vec{V}, \vec{L}_0, \vec{L}, S_0, \vec{R}_0$ and $\vec{u}$ obtained in Section \ref{latty1}. A major achievement of this paper is the existence of an invertible operator $\vec{\mathcal{P}}$ used in creating new  2-forms: $\vec{L}_1, S$ and $\vec{R}$ which are estimated in Subsection \ref{toute}. Subsection \ref{linkage} answers the important question of how to link the regularity of the primitives $S$ and $\vec{R}$ back to geometric quantities like $\vec{H}$ or $\vec{\Phi}$. This is done via the return equation obtained in Section \ref{latty1}. By obtaining a Morrey-type estimate, the integrability of $d\vec{H}$ is improved, the bootstrap argument ensues and the proof is finished.

\subsection{Set up and preliminary results} \label{adzsq}
\begin{df}
Let $\Omega\subset \mathbb{R}^4$ be a ball. The Sobolev space $W^{k,p}(\Omega)$  of measurable functions from $\Omega$ into $\mathbb{R}^m$ is defined by
$$W^{k,p}(\Omega):=\left\{ f\in L^p(\Omega) \,:\, D^\alpha f\in L^p(\Omega)\,\, \forall\,\,\alpha \,\,\mbox{with} \,\,|\alpha|\leq k          \right\}   $$
with the norm
$$\norm{f}_{W^{k,p}(\Omega)}:= \sum_{|\alpha|\leq k}\norm{D^\alpha f}_{L^p(\Omega)}.$$
The dual of $W^{k,p}(\Omega)$ is $W^{-k,p}(\Omega)$. 

\noindent
The homogeneous Sobolev space $\dot W^{k,p}(\Omega)$ is the space of all $k$-weakly differentiable functions $f\in L^p(\Omega)$ such that $D^\alpha f\in L^p(\Omega)$ for all $|\alpha|=k$.

\noindent
The closure of $C_c^\infty(\Omega)$ in $W^{k,p}(\Omega)$ is denoted by $W_0^{k,p}(\Omega)$. 
\end{df}
\noindent
Suppose that $\vec{\Phi}\in W^{3,2}(\Omega)\cap W^{1,\infty}(\Omega)$ is a critical point of the energy $\mathcal{E}(\Sigma).$

\noindent
By hypothesis, the coefficients of the metric tensor given by $g_{ij}:=\nabla_i\vec{\Phi}\cdot\nabla_j\vec{\Phi}$ lie in the space $W^{2,2}\cap L^{\infty}$. We are interested in the regularity of Willmore-type manifolds with no branch point, that is the metric does not degenerate. The metric coefficients are uniformly bounded from above and below and satisfy, in a local coordinate chart $\Omega$ of $\Sigma$:
\begin{gather}
g_{ij} \simeq \delta_{ij} \quad\quad \textnormal{as (4 $\times$ 4)-matrices.}  \nonumber
\end{gather}
In other words, $g_{ij}$ satisfies on $\Omega$
\begin{gather}
c^{-1} \delta_{ij} u^i v^j\leq g_{ij} u^i v^j \leq c\delta_{ij} u^i v^j     \nonumber
\end{gather}
for some $c>0$ and for all $u,v \in\mathbb{R}^4$. We have used $\delta_{ij}$ to denote the $(i,j)-$component of the flat metric.

\noindent
For $A\in \Lambda^k(\Omega)$ and $B\in \Lambda^k(\Omega)$ with local representations
$$A=\frac{1}{k!}A_{[i_1....i_k]} dx^{i_1}\wedge_4\cdots\wedge_4 dx^{i_k}\quad\mbox{and}\quad B=\frac{1}{k!}B_{[j_1....j_k]} dx^{j_1}\wedge_4\cdots\wedge_4 dx^{j_k} $$
we define the usual inner product on $\Omega\in\mathbb{R}^4 $
\begin{gather}
\langle A, B\rangle :=\frac{1}{(2k)!}\int_\Omega g^{i_1 j_1} \cdots g^{i_kj_k} A_{[i_1...i_k} B_{j_1...j_k]} dx^{1}\wedge_4 \cdots \wedge_4 dx^4.   \nonumber
\end{gather}
Since the metric is controlled by the Euclidean metric on $\Omega$, we have
\begin{gather}
\langle A, B\rangle \simeq  \langle A, B \rangle_0:=  \frac{1}{(2k)!}\int_\Omega \delta^{i_1 j_1} \cdots \delta^{i_kj_k} A_{[i_1...i_k} B_{j_1...j_k]} dx^{1}\wedge_4 \cdots \wedge_4 dx^4  \nonumber
\end{gather}
which is the usual inner-product on forms for the flat metric $\delta$. Accordingly, we have
\begin{align}
\norm{A}_{L^p}\equiv \sup_{\norm{B}_{L^{p'}}=1}\langle A, B\rangle _0 = \sup_{\norm{B}_{L^{p'}}=1}\langle A, B\rangle  \quad\quad\forall \,\,p\in[1,\infty)\,,\,\, \frac{1}{p} +\frac{1}{p'}=1.   \nonumber
\end{align}
\noindent
This fact will be used recurrently and tacitly in subsequent subsections, in  particular, when calling upon duality arguments.  
\\
First, we discuss the regularising properties of an equation of the type

\begin{align} \begin{cases} 
      \Delta u=f & \mbox{in}\,\, \Omega \\
      \,\,\,\,\,u=0 & \mbox{on} \,\, \partial \Omega.
   \end{cases} \label{take}
\end{align}

\noindent
Equivalently, we can write 
\begin{align}
\mathcal{L}[u]:= \partial_i(|g|^{1/2} g^{ij} \partial_j u)= |g|^{1/2} f  \nonumber
\end{align}
where $\mathcal{L}$ is a second-order partial differential operator. The coefficient $a^{ij}:\Omega\rightarrow \mathbb{R}$ defined by $a^{ij}:= |g|^{1/2} g^{ij}$ is clearly symmetric and 
satisfies the uniform ellipticity condition: there exists some $c > 0$ such that
\begin{align}
c^{-1}\delta^{ij} u_i v_j\leq a^{ij}u_iv_j \leq c \delta^{ij}u_iv_j\,, \,\,\forall\,\, u,v\in\mathbb{R}^4.\nonumber
\end{align}
\\
\noindent
Unfortunately, the coefficients $a^{ij}$ are not H\"older continuous and standard regularity theory tools for elliptic equations cannot be applied but thanks to  Claim \ref{cll}  and Proposition \ref{callit}, we are still able to obtain the desired results.

\begin{cl}\label{cll}
The coefficients $a^{ij}$ belong to the space $W^{2,2}\cap L^\infty$ under pointwise multiplication. 
\end{cl}

\noindent
By the Sobolev embedding theorem, $W^{2,2}\subset W^{1,4}$. Thus $a^{ij}$ belong to $L^\infty\cap W^{1,4}$.

\begin{rmk}
Although it is true that the product of a  function in $L^p$ and any metric component or its determinant remains in $L^p$, the same cannot be said of Sobolev spaces (or of their duals). This is because the metric coefficients are not sufficiently differentiable (not Lipschitz). But the coefficients belong to $L^\infty\cap W^{2,2}$ by Claim \ref{cll} and the following holds.
\end{rmk}

\begin{cl} \label{claimclaim}
Let $\alpha$ be a function in $(L^\infty \cap W^{2,2})$. Let $\beta$ be a function in  $\dot{W}^{-1,2}$. Then the product $\alpha\beta$ belongs to $  \dot{W}^{-1,2} \oplus L^{4/3} \subset \dot{W}^{-1,2}$.
\end{cl}

\noindent
As a consequence, we have  the norm
$$\norm{A}_{W_0^{1,2}}\equiv \sup_{\norm{B}_{\dot W^{-1,2}}=1} \langle A, B\rangle _0 \simeq \sup_{\norm{B}_{\dot W^{-1,2}}=1} \langle A, B\rangle .$$

\begin{prop}\label{callit}
Let $\Omega\subset\mathbb{R}^4$ be a ball. Suppose that  $f:\Omega\rightarrow \mathbb{R}$  with $f\in L^2(\Omega)$.  Let $u$ be a weak solution of the problem 
\[ \begin{cases} 
      \Delta u=f & \mbox{in}\,\, \Omega \\
      \,\,\,\,\,u=0 & \mbox{on} \,\, \partial \Omega.
   \end{cases}
\]
Then $D^2 u\in L^2$ and 
$$||D^2 u||_{L^2(\Omega)}\lesssim ||f||_{L^2(\Omega)}.$$
\end{prop}

\begin{proof}
Clearly $\mathcal{L}[u]=|g|^{1/2}f$ belongs to $L^2$ since $f\in L^2(\Omega)$. A classical result of Miranda \cite{mir} (see also Corollary 1.4 in \cite{cruz}) then yields the desired estimate
$$\norm{D^2 u}_{L^2(\Omega)}\lesssim \norm{|g|^{1/2} f}_{L^2(\Omega)}\lesssim \norm{f}_{L^2(\Omega)}.$$
\end{proof}

\noindent
Suppose that  $u\in \Lambda^k(\mathbb{R}^4)$  and $f=d^\star B$, where $B\in \Lambda^{k+1}(\mathbb{R}^4)$, then    \eqref{take}  has the special structure   
$$\Delta u=d^\star B$$
which is equivalent to
\begin{align}
\mathcal{L}[u_{i_1\cdots i_k}] \,\,dx^{i_1}\wedge_4\cdots\wedge_4 dx^{i_k}&=|g|^{1/2} d^\star B  \nonumber
\\&= \partial_a(|g|^{1/2} g^{ai_{k+1}}B_{i_1\cdots i_{k+1}})\,\,dx^{i_1}\wedge_4\cdots\wedge_4 dx^{i_k}.  \nonumber
\end{align}
\noindent
This is a decoupled system of $k$-uniformly elliptic equations in flat divergence form with coefficients in $W^{2,2}\cap L^\infty$ which embeds in $VMO\cap L^\infty$. Such equations are studied by Di Fazio in  \cite{faz} (see also \cite{auscher}).

\begin{prop}  \label{adz}
Let $\Omega\subset\mathbb{R}^4$ be a ball. Suppose that $B\in \Lambda^{k+1}\otimes L^p(\Omega)$ for $p\in(1,\infty)$. 
Let $u\in \Lambda^k(\Omega)$ be the $k$-form satisfying
\[ \begin{cases} 
      \Delta u=d^\star B& \mbox{in}\,\, \Omega \\
      \,\,\,\,\,u=0 & \mbox{on} \,\, \partial \Omega.
   \end{cases}
\]
Then 
$$\norm{Du}_{L^p(\Omega)}\lesssim \norm{B}_{L^p(\Omega)}.$$
\end{prop}

\begin{proof}
Using Theorem 2.1 in \cite{faz}, each component of $u$ satisfies
$$\norm{Du_{i_1\cdots i_k}}_{L^p(\Omega)}\lesssim \norm{|g|^{1/2} g^{ai_{k+1}}B_{i_1\cdots i_{k+1}}}_{L^p(\Omega)}  \lesssim \norm{B}_{L^p(\Omega)}.$$
\end{proof}

\noindent
We will require the following corollaries of Theorem 2.1 in \cite{faz}.
\begin{prop} \label{111}
Let $\Omega\in\mathbb{R}^4$ be a ball. Suppose that $u\in\Lambda^2(\Omega)$ satisfies the conditions
$$d^\star u=0\quad\quad\mbox{and}\quad\quad du\in \dot{W}^{-1,2}(\Omega).$$
Then
$$\norm{u}_{L^2(\Omega)}\lesssim \norm{du}_{\dot{W}^{-1,2}(\Omega)}.$$
\end{prop}

\begin{proof}
Let $\Psi \in \Lambda^2\otimes L^2(\Omega)$ be arbitrary. Standard Hodge theory guarantees  there exist $\psi\in \Lambda^3\otimes W^{1,2}(\Omega)$ and $\theta\in\Lambda^1\otimes W^{1,2}(\Omega)$ such that
$$\Psi =d\theta +d^\star \psi  \quad\quad\mbox{with}\quad\quad d \psi=0\quad\mbox{and}\quad d^\star\theta=0.$$ 
We have, per Claim \ref{claimclaim},
$$\mathcal{L}[\psi]=|g|^{1/2}\Delta\psi=|g|^{1/2} d\Psi \,\,\in \,\dot{W}^{-1,2}(\Omega).       $$
Reasoning componentwise (set $\psi|_{\partial \Omega}=0$), we use \cite{faz} to arrive at
\begin{align}
\norm{D\psi}_{L^2(\Omega)}\lesssim \norm{|g|^{1/2} d\Psi}_{\dot{W}^{-1,2}(\Omega)}\lesssim \norm{\Psi}_{L^2(\Omega)}.  \label{cklh}
\end{align}
Set $\theta|_{\partial \Omega}=0$. Observe that
\begin{align}
\langle u, \Psi\rangle &=\langle u, d\theta +d^\star \psi\rangle     \nonumber
\\&= \langle d^\star u, \theta \rangle + \langle u, \theta \rangle _{\partial \Omega} +\langle d u, \psi \rangle + \langle u, \psi \rangle _{\partial \Omega}  \nonumber
\\&= \langle du, \psi\rangle     \nonumber
\\&\lesssim \norm{du}_{\dot{W}^{-1,2}(\Omega)} \norm {\psi}_{W_0^{1,2}(\Omega)}   \nonumber
\\ &\overset{\eqref{cklh}}\lesssim \norm{du}_{\dot{W}^{-1,2}(\Omega)} \norm {\Psi}_{L^2(\Omega)}   .  \nonumber
\end{align}
Hence the result
$$\norm{u}_{L^2(\Omega)}\lesssim \norm{du}_{\dot{W}^{-1,2}(\Omega)}.$$

\end{proof}

\begin{prop} \label{nadaa}
Let $\Omega \subset \mathbb{R}^4$ be a ball. Any $Q\in\Lambda^3(\Omega, \Lambda^2(\mathbb{R}^5))\otimes \dot{W}^{-1,2}(\Omega)$ can be decomposed in the form
$$Q=dA+d^\star B \quad\quad\mbox{with}\quad\quad d^\star A=0\,,\,dB=0\,\,,\,\, A\in\Lambda^2\,\,,\,\, B\in\Lambda^4\,$$
and
$$\norm{A}_{L^2(\Omega)}+\norm{B}_{L^2(\Omega)}\lesssim \norm{Q}_{\dot{W}^{-1,2}(\Omega)}.$$
\end{prop}
\begin{proof}

Let $\Psi\in L^2(\Omega)$ be arbitrary. As in the proof of Proposition \ref{111}, we have
$$\Psi=d\theta+d^\star \psi\quad\mbox{with} \quad d^\star\theta=0\,,\,d\psi=0,\,\, \theta|_{\partial \Omega}=0, \,\, \psi|_{\partial \Omega}=0,$$
and
\begin{align}\norm{D\psi}_{L^2(\Omega)}\lesssim \norm{|g|^{1/2} d\Psi}_{\dot{W}^{-1,2}(\Omega)}\lesssim \norm{\Psi}_{L^2(\Omega)}.  \label{neededd}
\end{align}
Since $\star B$ is a 0-form and $d\psi=0$, we have  
\begin{align}
\langle \star B,\Psi \rangle&=\langle \star B, d\theta+d^\star\psi \rangle  \nonumber
\\&=\langle \star d^\star B, \psi \rangle+\langle \star  B, \theta \rangle_{\partial \Omega}+\langle \star  B, \star\psi \rangle_{\partial \Omega} \nonumber
\\&= \langle \star Q- d^\star(\star A), \psi \rangle  \nonumber
\\&= \langle \star Q, \psi \rangle  -\langle \star A, d\psi \rangle  -\langle \star A, \psi \rangle_{\partial\Omega} \nonumber
\\&\lesssim \norm{\star Q}_{\dot{W}^{-1,2}(\Omega)} \norm{\psi}_{W_0^{1,2}(\Omega)}   \nonumber
\\&\overset{\eqref{neededd}} \lesssim \norm{Q}_{\dot{W}^{-1,2}(\Omega)} \norm{\Psi}_{L^2(\Omega)} \nonumber
\end{align}
where Claim \ref{claimclaim} has been used to obtain the last line. Hence
\begin{align}\norm{B}_{L^2(\Omega)}\simeq \norm{\star B}_{L^2(\Omega)}\lesssim \norm{Q}_{\dot{W}^{-1,2}(\Omega)}.\label{jknba}\end{align}
 Similarly, we have
\begin{align}
\langle A, \Psi \rangle&=   \langle A, d\theta+ d^\star \psi \rangle   \nonumber
\\&= \langle dA, \psi \rangle  \nonumber
\\&= \langle Q-d^\star B, \psi \rangle  \nonumber
\\&= \langle Q, \psi \rangle  \nonumber
\\&\lesssim \norm{Q}_{\dot{W}^{-1,2}(\Omega)} \norm{\psi}_{W_0^{1,2}}  \nonumber
\\& \lesssim \norm{Q}_{\dot{W}^{-1,2}(\Omega)} \norm{\Psi}_{L^2(\Omega)}. \nonumber  
\end{align}
Hence 
\begin{gather}
\norm{A}_{L^2(\Omega)}\lesssim \norm{Q}_{\dot{W}^{-1,2}(\Omega)}.\label{jknb} 
\end{gather}
Combining the estimates \eqref{jknba} and \eqref{jknb} proves the result.
\end{proof}

We now state the following Bourgain-Brezis type result (see Theorem 6.4 in \cite{curca}, \cite{maz} and \cite{van})

\begin{prop} \label{curca}
Let $\Omega \subset \mathbb{R}^4$ be a ball and let $V\in\Lambda^1\otimes (L^1\oplus \dot{W}^{-2,2})$ with $d^\star V=0.$ There exists $w\in\Lambda^2\otimes \dot{W}^{-1,2}(\Omega)$ such that
$$d^\star w=V\quad,\quad dw=0\quad,\quad \norm{w}_{\dot{W}^{-1,2}(\Omega)}\lesssim \norm{V}_{L^1\oplus \dot{W}^{-2,2}(\Omega)}.$$
\end{prop}

\noindent
Finally, we will need the following technical lemma.

\begin{lem}  \label{tech}
Let $\Omega$ be a ball in $\mathbb{R}^4$. For $k\in (0,1)$, we denote by $\Omega_k$, the ball with same center as $\Omega$ and with radius rescaled by $k$. Suppose $T\in \Lambda^1\otimes L^2(\Omega)$ satisfies
$$\norm{dT}_{L^{4/3}(\Omega)}+\norm{d^\star T}_{L^{4/3}(\Omega) }\lesssim M,$$
for some constant $M>0$. Then
$$\norm{T}_{L^2(\Omega_k)}\lesssim M+|\Omega|^{1/4}\norm{T}_{L^2(\Omega)}.$$ 
\end{lem}

\begin{proof}
Let $\mu$ be a standard smooth cut-off function with
$$\mu|_{\Omega_k}=1\quad,\quad \mu|_{\mathbb{R}^4\setminus \Omega}=0\quad, \quad 0<\mu(x) <1\,\,\, \forall\,\, x   \in\mathbb{R}^4.         $$
Clearly, 
\begin{align}
\norm{d(\mu T)}_{L^{4/3}(\Omega)} &\lesssim \norm{d\mu}_{L^\infty (\Omega)}  \norm{T}_{L^{4/3}(\Omega)} +\norm{\mu}_{L^\infty (\Omega)} \norm{dT}_{L^{4/3}(\Omega)}  \nonumber
\\&\lesssim |\Omega|^{1/4}\norm{T}_{L^2(\Omega)} +M\,,          \label{1b1}
\end{align}
where we have used Jensen's inequality. By the same token, using
$$d^\star (\mu T)=\star(d\mu\wedge_4 \star T)+\mu d^\star T\,,$$
we find
\begin{align}
\norm{d^\star (\mu T)}_{L^{4/3}(\Omega)}  \lesssim |\Omega|^{1/4} \norm{T}_{L^2(\Omega)} +M\,.         \label{1b2}
\end{align}

\noindent
Next, let $\Psi \in \Lambda^1\otimes L^2(\Omega)$ be arbitrary. Consider the problem (componentwise)

$$\Delta \psi=\Psi\quad, \quad \psi|_{\partial \Omega}=0\quad,\quad \psi\in \Lambda^1\,.$$
Per Proposition \ref{callit} and the Sobolev embedding theorem, we have
\begin{align}
\norm{d\psi}_{L^4(\Omega)} +\norm{d^\star \psi}_{L^4(\Omega)} \lesssim \norm{D^2\psi}_{L^2(\Omega)}\lesssim \norm{\Psi}_{L^2(\Omega)}.   \label{1b3}
\end{align}
Next, we compute
\begin{align}
\langle \mu T, \Psi \rangle &= \langle \mu T, \Delta \psi\rangle \nonumber
\\&= \langle \mu T, dd^\star \psi +d^\star d\psi\rangle  \nonumber
\\& =\langle d^\star (\mu T), d^\star\psi  \rangle +\langle d(\mu T), d\psi\rangle  + \langle d (\mu T), \psi \rangle_{\partial\Omega} +\langle d^\star (\mu T), \psi \rangle_{\partial\Omega} \nonumber
\\&\overset{\eqref{1b1}, \eqref{1b2}, \eqref{1b3}} \lesssim (|\Omega|^{1/4}\norm{T}_{L^2(\Omega)} +M)  \norm{\Psi}_{L^2(\Omega)}.  \nonumber
\end{align}
This shows that
$$\norm{\mu T}_{L^2(\Omega)}\lesssim |\Omega|^{1/4} \norm{T}_{L^2(\Omega)} +M\,.$$
Since $\mu\equiv 1$ on $\Omega_k,$ the desired conclusion follows.

\end{proof}

 The energy $\mathcal{E}(\Sigma)$ cannot be assumed to be small since it is not bounded from below. Instead, we rescale our domain so as to obtain a ball $B\subset \mathbb{R}^4$ of arbitrary center and radius with
$$\vec{\norm{h}}_{L^4(B)} < \varepsilon,$$
\noindent
where $\varepsilon >0$ may be chosen as small as we please.
\noindent
A particular useful quantity  is the 2-vector-valued 2-form
$$\vec{\eta}:=d\vec{\Phi}\overset{\wedge}\wedge_4 d\vec{\Phi}.$$

\begin{prop}
It holds
\begin{align}
\norm{d^\star\vec{\eta}}_{L^4(B)} <\varepsilon  \label{ssssj}
\end{align}\end{prop}
\begin{proof}
It suffices to realise that
$$d^\star \vec{\eta}= (4\vec{H}\wedge \nabla_j\vec{\Phi}+\nabla^i\vec{\Phi}\wedge\vec{h}_{ij})\,\, dx^j$$
and $\vec{h}\in L^4(B)$, $d\vec{\Phi}\in L^\infty(B)$ and $||\vec{H}||_{L^4(B)}\lesssim ||\vec{h}||_{L^4(B)}\leq \varepsilon$.
\end{proof}

\subsection{The creation of three characteristic two-forms}  \label{amena}

Let $$E(\Omega):=\norm{DH}_{L^2(\Omega)} +\norm{H}_{L^4(\Omega)} <\infty.$$
We have seen in section \ref{latty1} using translation invariance of $\mathcal{E}(\Sigma)$ that there exists a divergence free tensor $\vec{V}$, that is $\nabla_j\vec{V}^j=\vec{0}$. Accordingly, we need to solve
$$d\star\vec{L}_0=\star\vec{V}$$
for $\vec{L}_0\in \Lambda^2(\Omega, \mathbb{R}^5)$ and understand the regularity of $\vec{L}_0$. We will call upon Proposition \ref{curca}, but we need to first understand the regularity of $\vec{V}$. The exact expression for  $\vec{V}$  (see Theorem \ref{cons-law}) reveals that
$$\vec{V}^j=\nabla^j(\Delta_\perp\vec{H})+\vec{a}^j+\nabla_k(\vec{b}^{kj})+\vec{f}^j\,\,,$$
with
\begin{align}
\norm{\vec{a}}_{L^1(\Omega)} +\vec{\norm{b}}_{L^{4/3}(\Omega)}   \lesssim E(\Omega)    \label{adq}
\end{align}
and 
$$\vec{f}^j\simeq (\vec{h}^{ij}\cdot\Delta_\perp\vec{H})\nabla_i\vec{\Phi}.$$
We can write 
$$\vec{f}=\nabla_k(h^{ij}\nabla^kH\nabla_i\vec{\Phi})-\vec{h}^{ij}h_{ik}\nabla^kH-
(\nabla^kH\nabla_kh^{ij})\nabla_i\vec{\Phi},$$
so that
$$\vec{\norm{f}}_{L^1\oplus \dot W^{-1,4/3}(\Omega)}\lesssim E(\Omega).$$
Altogether, we find
$$\vec{\norm{V}}_{\dot W^{-2,2}\oplus L^1(\Omega)}\lesssim E(\Omega),$$
where we have used that $\Delta_\perp \vec{H}\in \dot W^{-1,2}$  and the fact that $\dot W^{-1,4/3} \subset \dot W^{-2,2}$ by the dual of the Sobolev injection. Now, we call upon Proposition \ref{curca} to obtain $\vec{w}\in \Lambda^2\otimes \dot W^{-1,2}(\Omega)$ with $d^\star \vec{w}=\vec{V}$
and
\begin{align}
\vec{\norm{w}}_{\dot W^{-1,2}(\Omega)}\lesssim \vec{\norm{V}}_{\dot W^{-2,2}\oplus L^1(\Omega)}\lesssim E(\Omega).  \label{dfd}
\end{align}
We then set $\vec{L}_0:=\vec{w}$, so that $d\star\vec{L}_0=\star \vec{V}$.  From \eqref{dfd} we have that $\vec{L}_0\in \Lambda^2\otimes \dot W^{-1,2}(\Omega)$. Per Claim \ref{claimclaim},  we know that the operator $\star$ preserves  $\dot W^{-1,2}$ so that $\star\vec{L}_0\in \Lambda^2\otimes \dot W^{-1,2}(\Omega)$ with the estimate
\begin{align}
\vec{\norm{\star L_0}}_{\dot W^{-1,2}(\Omega)} \lesssim E(\Omega).  \label{tenn}
\end{align}

\noindent
Defining as in Proposition \ref{coro} the vector-valued 2-form
$$\star\vec{L}=\vec{L}_0-dH^2\wedge_4 d\vec{\Phi},$$
we see that $\vec{L}\in \dot W^{-1,2}$ with $d\vec{L}=d\star \vec{L}_0$ and the estimate
\begin{align}\vec{\norm{L}}_{\dot W^{-1,2}(\Omega)}\lesssim \vec{\norm{\star L_0}}_{\dot W^{-1,2}(\Omega)} +\norm{H}_{L^4(\Omega)}\norm{\nabla H}_{L^2(\Omega)}\overset{\eqref{tenn}}\lesssim  E(\Omega),       \label{ax}
\end{align}
where we have used that $L^{4/3}$ injects continuously into $\dot W^{-1,2}$ which is a consequence of the Sobolev injection $W_0^{1,2}\subset L^4$. 
\noindent
As in Proposition \ref{coro}, there exists 2-forms $S_0$ and $\vec{R}_0$ satisfying
$$dS_0=\vec{L}\overset{.}\wedge_4 d\vec{\Phi} \quad\mbox{and}\quad d\vec{R}_0= \vec{L}\overset{\wedge}\wedge_4 d\vec{\Phi}+\star d\vec{u}-(\vec{J}+8H^2\vec{H})\wedge \star d\vec{\Phi}$$
where $\vec{J}:=\frac{1}{2}\Delta_\perp \vec{H}+\frac{1}{2}|\vec{h}|^2\vec{H} -7|\vec{H}|^2\vec{H}$.

\noindent
In addition, we are free to demand $d^\star S_0=0$ and $d^\star\vec{R}_0=\vec{0}$.
\\
Next is to understand the regularity of $\vec{R}_0$ and $S_0$. Recalling the definition of $\vec{J}$ (see Proposition \ref{coro}) and the fact that $\Delta_\perp\vec{H}\in \dot W^{-1,2}$, we find that $\vec{J}\in  \dot W^{-1,2}$. It remains to study the regularity of $d\vec{u}$.
To this end we have defined (see Proposition \ref{coro})
$$\frac{5}{3}H^2 d^\star\vec{\eta}= d\vec{u} +d^\star\vec{v},\quad\quad \vec{u}\in\Lambda^0(\Omega, \Lambda^2(\mathbb{R}^5)).$$
Clearly $d^\star\vec{u}=\vec{0}$.  Let $\vec{\Theta}\in L^2(\Omega)$ be arbitrary. Consider the problem
$$\Delta \vec{\theta}=\vec{\Theta} \quad \mbox{on}\,\, \Omega\quad\mbox{with}\,\,\, \vec{\theta}|_{\partial \Omega}=\vec{0}.$$

\noindent
Proposition \ref{callit} confirms that
\begin{gather}
\vec{\norm{D\theta}}_{L^4(\Omega)}  \lesssim \vec{\norm{\theta}}_{W_0^{2,2}(\Omega)} \lesssim \vec{\norm{\Theta}}_{L^2(\Omega)}.  \label{klnb}
\end{gather}

\noindent
Observe that
\begin{align}
\langle \vec{\Theta}, \vec{u} \rangle&= \langle \Delta\vec{\theta}, \vec{u}\rangle =\langle d\vec{\theta}, d\vec{u} \rangle   \nonumber
\\&= \langle d\vec{\theta}, \frac{5}{3}H^2 d^\star \vec{\eta}- d^\star \vec{v} \rangle   \nonumber
\\&= \frac{5}{3}\langle d\vec{\theta}, H^2 d^\star\vec{\eta}\rangle  \nonumber
\\&\lesssim \norm{H^2 d^\star\vec{\eta}}_{L^{4/3}(\Omega)}  \vec{\norm{\theta}}_{L^4(\Omega)}      \nonumber
\\&\lesssim \norm{H^2 d^\star\vec{\eta}}_{L^{4/3}(\Omega)} \vec{\norm{\Theta}}_{L^2(\Omega)}.  \nonumber
\end{align}
This shows that
\begin{align}
\vec{\norm{u}}_{L^2(\Omega)} \lesssim \norm{H^2 d^\star\vec{\eta}}_{L^{4/3}(\Omega)}  \lesssim \norm{H}^2_{L^4(\Omega)} \norm{d^\star\vec{\eta}}_{L^4(\Omega)}  \overset{\eqref{ssssj}}\lesssim \varepsilon E(\Omega).  \label{zcx}
\end{align}

\noindent
We have observed from equation \eqref{hads8} that $\vec{u}\in \Lambda^0(\Omega, \mathbb{R}^5)$ satisfies
$$\Delta \vec{u}=\frac{5}{3}d^\star (H^2 d^\star \vec{\eta}),$$
which is of the type investigated in Proposition  \ref{adz}.
\\
We split $\vec{u}=\vec{u}_0+\vec{u}_1$ where 
\begin{equation*}
 \begin{cases}
           \Delta \vec{u}_0=\vec{0} & \mbox{in}\,\, \Omega \\
            \,\,\,\,\, \vec{u}_0= \vec{u} & \mbox{on} \,\, \partial \Omega,
       \end{cases} \quad
\quad\quad\mbox{and} \quad\quad\begin{cases}
            \Delta \vec{u}_1=\Delta\vec{u} & \mbox{in}\,\, \Omega \\
            \,\,\,\,\, \vec{u}_1= \vec{0} & \mbox{on} \,\, \partial \Omega,
       \end{cases}
\end{equation*}
The equation for $\vec{u}_1$ is handled as in Proposition \ref{adz}. We find
\begin{align}
\norm{D\vec{u}_1}_{L^{4/3}(\Omega)}  \lesssim \norm{H^2}_{L^2(\Omega)}\norm{d^\star\vec{\eta}}_{L^4(\Omega)} \overset{\eqref{ssssj}} \lesssim \varepsilon E(\Omega).   \label{vcb}
\end{align}
In particular, the Sobolev embedding theorem gives
\begin{align}
\norm{\vec{u}_1}_{L^2(\Omega)} \lesssim \varepsilon E(\Omega).  \nonumber
\end{align}
Hence by \eqref{zcx}
\begin{align}
\norm{\vec{u}_0}_{L^2(\Omega)} \lesssim \varepsilon E(\Omega)   . \label{pio}
\end{align}
On the other hand, the Caccioppoli inequality yields
\noindent
\begin{align}
\norm{D\vec{u}_0}_{L^{4/3}(\Omega ')} \lesssim \norm{D\vec{u}_0}_{L^2(\Omega ')} \lesssim \norm{\vec{u}_0}_{L^2(\Omega)} \overset{\eqref{pio}}\lesssim \varepsilon E(\Omega)  \quad\quad\forall\,\,\,\Omega '\Subset \Omega 
\end{align}

\noindent
Together with \eqref{vcb}, the latter gives
\begin{align}
\norm{D\vec{u}}_{L^{4/3}(\Omega ')}  \lesssim \varepsilon E(\Omega) \quad\quad \forall\,\,\, \Omega '\Subset \Omega.   \label{host}
\end{align}

\noindent
Using Claim \ref{claimclaim} and the estimates we have of the quantities $\vec{u}$, $\vec{J}$ and $\vec{L}$, we are now ready to state the first regularity property:
$$dS_0\,\, ,\,\, d\vec{R}_0\in\dot W^{-1,2}(\Omega)$$

\noindent
with the estimate
\begin{align}
\norm{dS_0}_{\dot W^{-1,2}(\Omega)} +||d\vec{R}_0|| _{\dot W^{-1,2}(\Omega)} \lesssim E(\Omega_1)    \quad\quad\forall\,\, \Omega\Subset \Omega_1.   \label{simm}
\end{align}

\noindent
Since $S_0$ and $\vec{R}_0$ are co-closed, by Proposition \eqref{111}, we obtain 
\begin{align}
\norm{S_0}_{L^2(\Omega)} +||\vec{R}_0||_{L^2(\Omega)} \lesssim E(\Omega_1) \quad\quad \Omega\Subset \Omega_1.  \label{ax1}
\end{align}

\begin{rmk}  \label{rama}
Note that
\begin{align}
\norm{d^\star\vec{v}}_{L^{4/3}(\Omega ')} \lesssim \norm{d\vec{u}}_{L^{4/3}(\Omega ')} + \norm{H}^2_{L^4(\Omega)} \norm{d^\star\vec{\eta}}_{L^4(\Omega)} \overset{\eqref{ssssj}, \eqref{host}} \lesssim \varepsilon E(\Omega_1).
\end{align}
\end{rmk}

\noindent
Our next goal is to obtain an estimate similar to \eqref{simm} but with $L^{4/3}$ norm in place of $\dot W^{-1,2}$. In order to do this, we first introduce an operator $\vec{\mathcal{P}}$ which helps to redefine $\vec{L}$.
Next, we redefine the two-forms $S_0$ and $\vec{R}_0$ in order to create $S$ and $\vec{R}$ respectively. It turns out that $dS$ and $d\vec{R}$ can then be estimated in $L^{4/3}$ as desired (this is done in Section \ref{toute}). The remaining part of the present section is devoted to introducing $\vec{\mathcal{P}}$ and redefining $\vec{L}$, $S_0$ and $\vec{R}_0$.

\begin{prop} \label{conty}
The operator 
$$\mathcal{\vec{P}}: \Lambda^1(\Sigma, \Lambda^1(\mathbb{R}^5))\longrightarrow \Lambda^2(\Sigma,\Lambda^2(\mathbb{R}^5))\quad\quad\mbox{defined by}\quad\quad \mathcal{\vec{P}}:\vec{\ell}\longmapsto \vec{\ell}\overset{\wedge}\wedge_4 d\vec{\Phi}$$
is algebraically invertible. Moreover \,$\mathcal{\vec{P}}^{-1}$ injects $L^2$ into itself continuously.
\end{prop}

\begin{proof}
We have 
$$\mathcal{\vec{P}}_{pq}:=\big(\mathcal{\vec{P}}(\vec{\ell})  \big)_{pq} =\vec{\ell}_p\wedge \nabla_q\vec{\Phi}-\vec{\ell}_q\wedge\nabla_p\vec{\Phi}.$$
\noindent
Let $\vec{n}$ be the normal vector to the Willmore-type hypersurface $\Sigma$, defined on a small patch $\Omega \subset \mathbb{R}^4$ by
$$\vec{n}:=\star_{\mathbb{R}^5}\star (d\vec{\Phi}\overset{\wedge}\wedge_4 d\vec{\Phi}\overset{\wedge}\wedge_4 d\vec{\Phi}\overset{\wedge}\wedge_4 d\vec{\Phi})\in \Lambda^0(\Omega, \Lambda^1(\mathbb{R}^5))$$
where $\star_{\mathbb{R}^5}$ acts on the 4-vectors in the ambient space while $\star$ acts on the 4-forms in the parameter space.

\noindent
We compute
\begin{align}
\mathcal{\vec{P}}_{pq}\cdot (\vec{n}\wedge \nabla^q\vec{\Phi})&= (\vec{\ell}_p\wedge \nabla_q\vec{\Phi}-\vec{\ell}_q\wedge\nabla_p\vec{\Phi})\cdot (\vec{n}\wedge\nabla^q\vec{\Phi})             \nonumber
\\&= 3\vec{n}\cdot\vec{\ell}_p\,, \nonumber
\end{align}
so that
\begin{gather}
\vec{\eta}\cdot\vec{\ell}_p =\frac{1}{3}\mathcal{\vec{P}}_{pq}\cdot (\vec{n}\wedge\nabla^q\vec{\Phi}).   \label{ta}
\end{gather}

\noindent
On the other hand, we have
\begin{align}
\mathcal{\vec{P}}_{pq}\cdot (\nabla^p\vec{\Phi}\wedge \nabla^q\vec{\Phi})&= (\vec{\ell}_p\wedge \nabla_q\vec{\Phi}-\vec{\ell}_q\wedge \nabla_p\vec{\Phi})\cdot (\nabla^p\vec{\Phi}\wedge \nabla^q\vec{\Phi})   \nonumber
\\&= 6\vec{\ell}_p\cdot\nabla^p\vec{\Phi}\,,   \nonumber
\end{align}
so that
\begin{gather}
\vec{\ell}_p\cdot\nabla^p\vec{\Phi} =\frac{1}{6} \mathcal{\vec{P}}_{pq} \cdot (\nabla^p\vec{\Phi}\wedge \nabla^q\vec{\Phi}).     \label{sothat}
\end{gather}

\noindent
Finally, we check that
\begin{align}
\mathcal{\vec{P}}_{pq}\cdot (\nabla^i\vec{\Phi}\wedge \nabla^q\vec{\Phi}) &= (\vec{\ell}_p\wedge \nabla_q\vec{\Phi}-\vec{\ell}_q\wedge \nabla_p\vec{\Phi})\cdot (\nabla^i\vec{\Phi}\wedge \nabla^q\vec{\Phi})   \nonumber
\\&= 2\vec{\ell}_p\cdot \nabla^i\vec{\Phi} +\delta^i_p \vec{\ell}_q\cdot \nabla^q\vec{\Phi}   \nonumber
\\&\overset{\eqref{sothat}}=2\vec{\ell}_p\cdot \nabla^i\vec{\Phi} +\frac{1}{6}\delta^i_p \mathcal{\vec{P}}_{sq} \cdot (\nabla^s\vec{\Phi}\wedge \nabla^q\vec{\Phi}),  \nonumber
\end{align}
so that
\begin{align}
\vec{\ell}_p\cdot \nabla^i\vec{\Phi} =\frac{1}{2}\mathcal{\vec{P}}_{pq}\cdot (\nabla^i\vec{\Phi}\wedge \nabla^q\vec{\Phi})  -\frac{1}{12}\delta^i_p \mathcal{\vec{P}}_{sq}\cdot (\nabla^s\vec{\Phi}\wedge \nabla^q\vec{\Phi}).    \label{ta1}
\end{align}

\noindent
According to \eqref{ta} and \eqref{ta1}, the 1-form $\vec{\ell}$ can be totally recovered from $\mathcal{\vec{P}}$ via
$$\vec{\ell}_p=\left[\frac{1}{2}\mathcal{\vec{P}}_{pq}\cdot (\nabla^i\vec{\Phi}\wedge \nabla^q\vec{\Phi})  -\frac{1}{12}\delta^i_p \mathcal{\vec{P}}_{sq}\cdot (\nabla^s\vec{\Phi}\wedge \nabla^q\vec{\Phi})  \right]\nabla_i\vec{\Phi} +\left[\frac{1}{3}\mathcal{\vec{P}}_{pq}\cdot (\vec{n}\wedge\nabla^q\vec{\Phi})  \right]\vec{n}.$$
That $\mathcal{\vec{P}}^{-1}$ maps $L^2$ onto itself continuously is trivial.

\end{proof}

\noindent
Observe that given any $\vec{\mu}\in\Lambda^2(\Omega, \Lambda^2(\mathbb{R}^5))\otimes L^2(\Omega)$, we can define
\begin{align}
\vec{\ell}:= \mathcal{\vec{P}}^{-1}(\vec{\mu}).   \label{jert}
\end{align}

\noindent
We now create a new 2-form 
$$\vec{L}_1:=\vec{L}+ d\vec{\ell}$$
satisfying
\begin{gather}
\vec{\norm{L_1}}_{\dot W^{-1,2}(\Omega)}  \lesssim \vec{\norm{L}}_{\dot W^{-1,2}(\Omega)} + \vec{\norm{\ell}}_{L^2(\Omega)}  \overset{\eqref{ax}} \lesssim E(\Omega_1) +||\vec{\mu}||_{L^2(\Omega)}.               \nonumber
\end{gather}

\noindent
Consider now 
\begin{align}
S:=S_0+\vec{\ell}\overset{.}\wedge_4 d\vec{\Phi} \quad\quad\mbox{and}\quad\quad \vec{R}:=\vec{R}_0+\vec{\ell}\overset {\wedge}\wedge_4 d\vec{\Phi}. \label{merio}
\end{align}
We immediately see that
$$dS=\vec{L}_1\overset{.}\wedge_4 d\vec{\Phi}\quad\mbox{and}\quad d\vec{R}= \vec{L}_1\overset{\wedge}\wedge_4 d\vec{\Phi} +\star d\vec{u} -(\vec{J} +8H^2\vec{H})\wedge\star d\vec{\Phi}.$$

\subsection{Estimating $S$ and $\vec{R}$ in the correct space}  \label{toute}

From section \ref{latty1}, Proposition \ref{coro}, the two-forms $S_0$ and $\vec{R}_0$ satisfy 
$$dS_0= \vec{L}\overset{\cdot}\wedge_4 d\vec{\Phi}\quad\quad\mbox{and}\quad\quad d\vec{R}_0= \vec{L}\overset{\wedge}\wedge_4 d\vec{\Phi} +\star d\vec{u}-\vec{J}_0\wedge\star d\vec{\Phi}$$
where for convenience, we have set $\vec{J}_0:= \vec{J}+8H^2\vec{H}.$

\noindent
We compute
\begin{align}
\quad\star\vec{\eta}\overset{\cdot}\wedge_4\star(d\vec{R}_0-\star d\vec{u} +\vec{J}_0\wedge\star d\vec{\Phi})  \nonumber
& =\epsilon_{ijk[r} \epsilon_{ab]pq} \vec{L}^{ij}\cdot\nabla^p\vec{\Phi} \, g^{kq}\,\,\, dx^{a}\wedge_4 dx^b\wedge_4 dx^r  \nonumber
\\&= \vec{L}_{[ab}\cdot\nabla_{r]}\vec{\Phi}  \,\,\,dx^a\wedge_4 dx^b\wedge_4 dx^r   \nonumber  
\\&=dS_0.  \nonumber
\end{align}

\noindent
As $\vec{J}_0$ is a normal vector, we have 
$$\star\vec{\eta}\overset{\cdot}\wedge_4 (\vec{J}_0\wedge \star d\vec{\Phi})=0.$$

\noindent
Hence 
\begin{align}
dS_0 =\star\vec{\eta} \overset{\cdot} \wedge_4 \star d\vec{R}_0 -\star\vec{\eta} \overset{\cdot}\wedge_4 d\vec{u}.  \label{rela}
\end{align}

\noindent
By a similar token, we find (see Appendix \ref{grace5} for the full computation)
\begin{align}
&\quad\star\vec{\eta} \overset{\bullet} \wedge_4 \star (d\vec{R}_0-\star d\vec{u} +\vec{J}_0\wedge \star d\vec{\Phi})    \nonumber
\\&=\epsilon_{ijk[r} \epsilon_{ab]pq} (\nabla^p\vec{\Phi}\wedge \nabla^q\vec{\Phi})  \bullet (\vec{L}^{ij}\wedge \nabla^k\vec{\Phi})  \,\,\, dx^a\wedge_4 dx^b\wedge_4 dx^r   \nonumber
\\&=\epsilon_{ijk[r} \epsilon_{ab]pq} \left[(\nabla^p\vec{\Phi}\cdot \vec{L}^{ij}) \nabla^q\vec{\Phi}\wedge \nabla^k\vec{\Phi} + g^{qk} \nabla^p\vec{\Phi}\wedge \vec{L}^{ij} \right]    \nonumber
\\&= \star\vec{\eta}\wedge_4 \star dS_0 -(d\vec{R}_0-\star d\vec{u} +\vec{J}_0\wedge \star d\vec{\Phi})  +\vec{Q}_0   \label{kli}
\end{align}

\noindent
where $\vec{Q}_0\in \Lambda^3(\Omega, \Lambda^2(\mathbb{R}^5))$ is given by
$$\vec{Q}_0:=\frac{1}{12}\delta^{\alpha\beta\gamma}_{abr}\left[(\vec{L}_{k\alpha}\cdot\nabla^k\vec{\Phi}) 
\nabla_\beta\vec{\Phi} -(\vec{L}_{k\alpha}\cdot\nabla_\beta\vec{\Phi})\nabla^k\vec{\Phi}  \right]\wedge \nabla_\gamma \vec{\Phi}\,\,\,\, dx^a\wedge_4 dx^b\wedge_4 dx^r.$$

\noindent
As done in Appendix \ref{grace5}, we find
$$\star\vec{\eta}\overset {\bullet}\wedge_4 (\vec{J}_0\wedge d\vec{\Phi})=-\vec{J}_0\wedge \star d\vec{\Phi}.$$

\noindent
Hence from equation \eqref{kli}, we have
\begin{align}
d\vec{R}_0=\star\vec{\eta}\overset{\bullet}\wedge_4 \star d\vec{R}_0+ \star\vec{\eta}\wedge_4 \star dS_0 -\star\vec{\eta}\overset{\bullet}\wedge_4 d\vec{u} -\star d\vec{u} +\vec{Q}_0.  \label{kakio}
\end{align}

\noindent
We continue with one trivial identity which holds by the antisymmetry of $(\star\vec{R}_0)$ and $(\star S_0)$,
\begin{align}
&\quad\quad(\nabla_l\vec{\eta}_{ij}+\nabla_j\vec{\eta}_{il})\bullet(\star\vec{R}_0)^{lj}   + (\nabla_l\vec{\eta}_{ij}+\nabla_j\vec{\eta}_{il})(\star S_0)^{lj} \nonumber
\\&=(\nabla_l\vec{\eta}_{ij}-\nabla_l\vec{\eta}_{ij})\bullet(\star\vec{R}_0)^{lj}   + (\nabla_l\vec{\eta}_{ij}-\nabla_l\vec{\eta}_{ij})(\star S_0)^{lj} \nonumber
\\&=\vec{0}.  \nonumber
\end{align}

\noindent
As $d\vec{\eta}=\vec{0}$, that is, $\nabla_i\vec{\eta}_{jl} +\nabla_j\vec{\eta}_{li} +\nabla_l\vec{\eta}_{ij}=\vec{0}$, the latter yields
\begin{align}
 (-\nabla_i\vec{\eta}_{jl}+2\nabla_j\vec{\eta}_{il})\bullet(\star\vec{R}_0)^{lj} +(-\nabla_i\vec{\eta}_{jl}+2\nabla_j\vec{\eta}_{il})(\star S_0)^{lj} =\vec{0}.  \label{lassz}
\end{align}

\noindent
We next compute
\begin{align}
&\quad\nabla^j(\vec{\eta}_{jl}\bullet (\star\vec{R}_0)^{il}+\vec{\eta}_{jl}(\star S_0)^{il})   \nonumber
\\&= \vec{\eta}_{jl}  \bullet \nabla^j(\star\vec{R}_0)^{il} + (d^\star\vec{\eta})_l \bullet(\star\vec{R}_0)^{il}    \nonumber
+\vec{\eta}_{jl}  \nabla^j(\star{S}_0)^{il} + (d^\star\vec{\eta})_l (\star{S}_0)^{il}     \nonumber
\\&= -\vec{\eta}_{jl}\bullet \big(\nabla^l(\star\vec{R}_0)^{ji} +\nabla^i(\star\vec{R}_0)^{lj}  \big)   +(d^\star\vec{\eta})_l\bullet (\star\vec{R}_0)^{il}    \nonumber
\\&\quad \quad  -\vec{\eta}_{jl} \big(\nabla^l(\star{S}_0)^{ji} +\nabla^i(\star{S}_0)^{lj}  \big)   + (d^\star\vec{\eta})_l  (\star{S}_0)^{il}  \label{okka}
\\&= -\nabla^l\big( \vec{\eta}_{jl} \bullet (\star \vec{R}_0)^{ji}  \big) +\nabla^l\vec{\eta}_{jl}\bullet (\star\vec{R}_0)^{ji} +(d^\star\vec{\eta})_l\bullet (\star\vec{R}_0)^{il}    \nonumber
\\&\quad\quad+\nabla^i\big(\vec{\eta}_{lj}\bullet(\star\vec{R}_0)^{lj}  \big)  - \nabla^i\vec{\eta}_{lj}\bullet(\star \vec{R}_0)^{lj}  \nonumber
\\&\quad\quad\quad
-\nabla^l\big( \vec{\eta}_{jl}  (\star {S}_0)^{ji}  \big) +\nabla^l\vec{\eta}_{jl}(\star S_0)^{ji} +(d^\star\vec{\eta})_{l}(\star S_0)^{il}  \nonumber
\\&\quad\quad\quad\quad+\nabla^i\big(\vec{\eta}_{lj}(\star{S}_0)^{lj}  \big)  -  \nabla^i\vec{\eta}_{lj} (\star {S}_0)^{lj}   \nonumber
\\&= -\nabla^j\big( \vec{\eta}_{jl}\bullet (\star\vec{R}_0)^{il} +\vec{\eta}_{jl} (\star S_0)^{il}  \big) + \nabla^i\big( \vec{\eta}_{jl}\bullet (\star\vec{R}_0)^{jl} +\vec{\eta}_{jl} (\star S_0)^{jl}  \big)  \nonumber
\\& \quad\quad - \nabla^i\vec{\eta}_{lj}\bullet (\star \vec{R}_0)^{lj}  -\nabla^i\vec{\eta}_{lj}(\star S_0)^{lj}  +2(d^\star\vec{\eta})_l\bullet (\star \vec{R}_0)^{il} +2(d^\star\vec{\eta})_l(\star {S}_0)^{il} \nonumber
\end{align}
where we have used that $S_0$ and $\vec{R}_0$ are co-closed to obtain \eqref{okka}.

\noindent
Hence
\begin{align}
\nabla^j\big( \vec{\eta}_{jl}\bullet (\star\vec{R}_0)^{il} +\vec{\eta}_{jl} (\star S_0)^{il}  \big)& = \frac{1}{2}\nabla^i\big( \vec{\eta}_{jl}\bullet (\star\vec{R}_0)^{jl} +\vec{\eta}_{jl} (\star S_0)^{jl}  \big)  \nonumber
\\&\quad -\frac{1}{2}\left(  \nabla^i\vec{\eta}_{lj}\bullet(\star \vec{R}_0)^{lj}  +\nabla^i\vec{\eta}_{lj}(\star S_0)^{lj}\right) \nonumber
\\&\quad\quad +(d^\star\vec{\eta})_l\bullet (\star \vec{R}_0)^{il} +(d^\star\vec{\eta})_l(\star {S}_0)^{il} .  \label{tyq}
\end{align}

\noindent
This yields
\begin{align}
&\quad\quad\left( \star(\star\vec{\eta}\overset{\bullet}\wedge_4 \star d\vec{R}_0 +\star\vec{\eta}\wedge_4 \star dS_0  )  \right)^i   \nonumber
\\&= \epsilon^{jkli}\frac{1}{2}\epsilon_{jkpq}\big( \vec{\eta}^{pq}\bullet \nabla^r(\star\vec{R}_0)_{rl} +\vec{\eta}^{pq} \nabla^r(\star S_0)_{rl}   \big)  \nonumber
   \\&  =2\vec{\eta}^{li}\bullet \nabla^j(\star \vec{R}_0)_{jl}  +2\vec{\eta}^{li} \nabla^j(\star S_0)_{jl}\nonumber
\\&=2 \nabla^j\left (\vec{\eta}^{li}\bullet (\star{R}_0)_{jl}  +\vec{\eta}^{li} (\star{S}_0)_{jl}      \right)   -2\nabla^j\vec{\eta}^{li} \bullet(\star\vec{R}_0)_{jl}   -2\nabla^j\vec{\eta}^{li}(\star S_0)_{jl}   \nonumber
\\&\overset{\eqref{tyq}}=2 \nabla^j\left (\vec{\eta}_{jl}\bullet (\star\vec{R}_0)^{il}-\vec{\eta}^{il}\bullet (\star{R}_0)_{jl}  
+\vec{\eta}_{jl} (\star {S}_0)^{il} -\vec{\eta}^{il} (\star{S}_0)_{jl}     \right)  \nonumber
\\&\quad\quad - \nabla^i\left( \vec{\eta}_{jl}\bullet (\star\vec{R}_0)^{jl}   + \vec{\eta}_{jl} (\star{S}_0)^{jl}\right) +\left(\nabla^i\vec{\eta}_{lj}\bullet (\star\vec{R}_0)^{lj} +\nabla^i\vec{\eta}_{lj} (\star{S}_0)^{lj}   \right) \nonumber
\\&\quad\quad\quad -2\nabla^j\vec{\eta}^{li}\bullet(\star\vec{R}_0)_{jl}       -2\nabla^j\vec{\eta}^{li}(\star{S}_0)_{jl}       -2(d^\star\vec{\eta})_l\bullet (\star \vec{R}_0)^{il} -2(d^\star\vec{\eta})_l (\star {S}_0)^{il}  \nonumber
\\&\overset{\eqref{lassz}}= \nabla^j\tensor{\vec{P}}{_j^i}  +\nabla^i\vec{G}  -4\nabla^j\vec{\eta}^{li}\bullet(\star\vec{R}_0)_{jl}       -4\nabla^j\vec{\eta}^{li}(\star{S}_0)_{jl}  -2(d^\star\vec{\eta})_l\bullet (\star \vec{R}_0)^{il} -2(d^\star\vec{\eta})_l (\star {S}_0)^{il}   \nonumber
\\&=\nabla^j\tensor{\vec{P}}{_j^i}  +\nabla^i\vec{G} +(\star\vec{Q}_1)^i,   \nonumber
\end{align}
where
\begin{align}
\vec{P}_{ij}:= \tensor{\vec{\eta}}{_{j}^l}\bullet \tensor{(\star\vec{R}_0)}{_{i}_l}-\tensor{\vec{\eta}}{_{i}^l}\bullet (\star{R}_0)_{jl}  
+\tensor{\vec{\eta}}{_{j}^l} (\star {S}_0)_{il} -\tensor{\vec{\eta}}{_{i}^l} (\star{S}_0)_{jl}
\end{align}
is the component of the two-vector valued 2-form $\vec{P}:=\vec{P}_{ij}\,\, dx^i\wedge_4 dx^j$, 
$$\vec{G}:= \vec{\eta}_{lj}\bullet (\star\vec{R}_0)^{jl}   + \vec{\eta}_{lj} (\star{S}_0)^{jl}   \,\,\,\in\Lambda^0(\mathbb{R}^4, \Lambda^2(\mathbb{R}^5))$$
and
$$(\star\vec{Q}_1)^i:=-4\nabla^j\vec{\eta}^{li}\bullet(\star\vec{R}_0)_{jl}       -4\nabla^j\vec{\eta}^{li}(\star{S}_0)_{jl}  -2(d^\star\vec{\eta})_l\bullet (\star \vec{R}_0)^{il} -2(d^\star\vec{\eta})_l (\star {S}_0)^{il}  .$$

\noindent
Introducing \eqref{kakio}, we find
\begin{align}
\star d\vec{R}_0=d^\star \vec{P} +d\vec{G} +\vec{U} +\star\vec{Q},    \label{zads}
\end{align}
where for notational convenience we have set
\begin{align}
\vec{U}:=\star(\star\vec{\eta}\overset{\bullet}\wedge_4 d\vec{u} +\star d\vec{u}),  \quad\quad\vec{Q}:=\vec{Q}_0+\vec{Q}_1.  \label{notational}
\end{align}

\noindent
By estimates \eqref{ax}, \eqref{ax1} and Claim \eqref{claimclaim}, we have
\begin{align}
||\vec{Q}||_{\dot W^{-1,2}(\Omega)}  \lesssim ||\vec{L}||_{\dot W^{-1,2}(\Omega)} +||\vec{R}_0||_{L^2(\Omega)} +||S_0||_{L^2(\Omega)}   \overset{\eqref{ax},\eqref{ax1}}\lesssim E(\Omega_1).\label{lives}
\end{align}

\noindent
Standard Hodge theory gives the decomposition
$$\vec{Q}=d\vec{A}+d^\star \vec{B}$$
$$\mbox{with}\quad d^\star\vec{A}=\vec{0}\,,\,d\vec{B}=\vec{0},  \quad\quad\quad \vec{A}\in\Lambda^2(\mathbb{R}^4, \Lambda^2(\mathbb{R}^5)), \vec{B}\in\Lambda^4(\mathbb{R}^4, \Lambda^2(\mathbb{R}^5)).$$

\noindent
We will now write $\vec{R}$ in terms of $\vec{G}$, $\star{B}$ and $\vec{U}$ and according to the definition \eqref{jert} of $\vec{\ell}$. 
Owing to Proposition  \ref{conty}, we define
$$\vec{\ell}:= \mathcal{P}^{-1}(-\vec{A}-\star\vec{P}).$$
\noindent
By Proposition \ref{nadaa}, we know that $\vec{A}$ belongs to $L^2$; by estimate \eqref{ax1} and equation \eqref{zads}, we know that $\star\vec{P}$ also belongs to $L^2$. Thus the sum $\vec{A}+\star\vec{P}$ belongs to $ L^2$ and now $\vec{\ell}$ is a suitable candidate to redefine $\vec{L}$.

\noindent
Accordingly, from \eqref{zads} we have
\begin{align}
d(\vec{G}+\star\vec{B}) +\vec{U} = \star d(\vec{R}_0-\vec{A}-\star\vec{P}) =\star d(\vec{R}_0+\vec{\ell}\overset{\wedge}\wedge_4 d\vec{\Phi}) \overset{\eqref{merio}}=\star d\vec{R} .   \label{hags}
\end{align}

\noindent
We finish this subsection with estimates. Let $\Omega_1$ appearing in \eqref{ax1}   be a ball $B_r$ of fixed radius $r$. Let $\vec{F}:= \vec{G}+\star\vec{B}$, then we have

\begin{lem}\label{delaa}
It holds
\begin{align}
||d\vec{F}||_{L^{4/3}(B_{kr})}\lesssim (k+\varepsilon)E(B_r)\quad\quad\forall\,\,0<k<1/2.    \nonumber
\end{align}
\end{lem}

\begin{proof}
Observe first that
$$\Delta \vec{F}=-d^\star \vec{U}.$$
To study this equation, note that
\begin{align}
||\vec{G}||_{L^2(B_{r/2})}\quad& \lesssim \quad||\vec{\eta}||_{L^\infty(B_{r/2})} \left(||\vec{R}_0||_{L^2(B_{r/2})} +||S_0||_{L^2(B_{r/2})}   \right)     \nonumber
\\\quad\quad&\overset{ \eqref{ax1}} \lesssim E(B_r) .  \label{gag1}
\end{align}

\noindent
By Proposition \eqref{nadaa}, we have
\begin{align}
||\star\vec{B}||_{L^2(B_{r/2})}\lesssim ||\vec{Q}||_{\dot W^{-1,2}(B_{r/2})} \overset{\eqref{lives}}\lesssim E(B_{r}).  \label{gaga}  \end{align}

\noindent
Combining the estimates \eqref{gag1} and \eqref{gaga}  yield
\begin{align}
||\vec{F}||_{L^2(B_{r/2})}  \lesssim E(B_{r}).  \label{amer1}
\end{align}

\noindent
We split $\vec{F}=\vec{F}_0+\vec{F}_1$   such that

\begin{equation*}
 \begin{cases}
           \mathcal{L}[ \vec{F}_0]=\vec{0} & \mbox{in}\,\, B_{r/2} \\
            \,\,\,\,\, \vec{F}_0= \vec{F} & \mbox{on} \,\, \partial B_{r/2},
       \end{cases} \quad
\quad\quad\mbox{and} \quad\quad\begin{cases}
            \mathcal{L} [\vec{F}_1]=|g|^{1/2}d^\star\vec{U} & \mbox{in}\,\, B_{r/2}\\
            \,\,\,\,\, \vec{F}_1= \vec{0} & \mbox{on} \,\, \partial B_{r/2}.
       \end{cases}
\end{equation*}

\noindent
Elliptic estimates (see Proposition \ref{adz} and recall  $\vec{F}_1$ is a 0-form) yield
\begin{align}
||\vec{F}_1||_{L^2(B_{r/2})}+ ||d\vec{F}_1||_{L^{4/3}(B_{r/2})}  \lesssim ||\vec{U}||_{L^{4/3}(B_{r/2})}.  \label{amer}
\end{align}

\noindent
On the other hand the Jensen and Caccioppoli inequalities give
\begin{align}
||d\vec{F}_0||_{L^{4/3}(B_{kr/2})}& \lesssim kr||d\vec{F}_0||_{L^2(B_{kr/2})}  \nonumber
\\&\lesssim kr||d\vec{F}_0||_{L^2(B_{r/2})}  \nonumber
\\&\lesssim k||\vec{F}_0||_{L^2(B_{r/2})}  \nonumber
\\&\lesssim k||\vec{F}||_{L^2(B_{r/2})} +k||\vec{F}_1||_{L^2(B_{r/2})}   \nonumber
\\&\overset{\eqref{amer}  \eqref{amer1}} \lesssim ||\vec{U}||_{L^{4/3}(B_{r/2})} +kE(B_{r}).   \nonumber
\end{align}

\noindent
Combining the latter with \eqref{amer} yields the estimate
\begin{align}
||d\vec{F}||_{L^{4/3}(B_{kr/2})} & \lesssim ||d\vec{F}_0||_{L^{4/3}(B_{kr/2})} +||d\vec{F}_1||_{L^{4/3}(B_{r/2})} \nonumber
\\&\lesssim ||\vec{U}||_{L^{4/3}(B_{r/2})} +kE(B_{r}).
\end{align}
It remains to observe from \eqref{notational} that
\begin{align}
||\vec{U}||_{L^{4/3}(B_{r/2})}\lesssim ||d\vec{u}||_{L^{4/3}(B_{r/2})}\overset{\eqref{host}} \lesssim \varepsilon E(B_r).   \label{starss}
\end{align}
Relabelling the domains and choosing $0<k<1/2$ rather than $0<k<1$, we obtain
$$||d\vec{F}||_{L^{4/3}(B_{kr})} \lesssim (k+\varepsilon) E(B_r).$$
This completes the proof.

\end{proof}

\begin{cor}\label{takeova}
We have
$$\norm{dS}_{L^{4/3}(B_{kr})} +||d\vec{R}||_{L^{4/3}(B_{kr})} \lesssim (k+\varepsilon) E(B_r) \quad\quad \forall\,\,k\in(0,1/2).$$
\end{cor}
\begin{proof}
It follows immediately from Lemma \ref{delaa}, equation \eqref{hags} and the estimate
\begin{align}
||\vec{U}||_{L^{4/3}(B_{kr})}\lesssim ||d\vec{u}||_{L^{4/3}(B_{kr})}\overset{\eqref{host}} \lesssim \varepsilon E(B_r).   
\end{align}
that
\begin{align}
||d\vec{R}||_{L^{4/3}(B_{kr})}  \lesssim (k+\varepsilon) E(B_r).  \label{messia}
\end{align}

\noindent
Recall that $dS$ and $d\vec{R}_0$ are defined from the modified $\vec{L}+d\vec{\ell}$. Thus $S$ and $\vec{R}$ form an admissible pair and are linked together by the same relation  \eqref{rela}
linking $S_0$ and $\vec{R}_0$, namely
$$dS=\star\vec{\eta}\overset{\cdot}\wedge_4\star d\vec{R} -\star\vec{\eta}\overset{\cdot}\wedge_4 d\vec{u}.$$

\noindent
Calling upon \eqref{messia} and  \eqref{host} gives the desired estimate
$$\norm{dS}_{L^{4/3}(B_{kr})}\lesssim  (k+\varepsilon) E(B_r).$$

\end{proof}
\noindent
The advantage of $S$ and $\vec{R}$ over $S_0$ and $\vec{R}_0$ is that Corollary \ref{takeova} involves $L^{4/3}$ norm rather than $\dot W^{-1,2}$ norm in the estimate \eqref{simm}.

\subsection{The return equation: controlling the geometry} \label{linkage}
We recall the return equation from section \ref{latty1} (precisely \eqref{recca})
$$12d^\star \vec{T}=\star (d\vec{R} \overset{\bullet}\wedge_4 d\vec{\Phi}) +\star (dS \wedge_4 d\vec{\Phi}) +6\star(\star d^\star \vec{v}\overset{ \bullet}\wedge_4 d\vec{\Phi})$$
where 
\begin{align}
\vec{T}:=(\nabla_j\vec{H}+2Hh_{jk}\nabla^k\vec{\Phi}-2H^2\nabla_j\vec{\Phi})dx^j   \label{nameT}
\end{align}
and $\vec{v}$ is as in Remark \ref{rama}. Note first that
\begin{align}
\vec{\norm{T}}_{L^2(B_{kr})} \lesssim \norm{dH}_{L^2(B_{kr})} +\norm{H}_{L^4(B_{kr})} \norm{h}_{L^4(B_{kr)}} \lesssim E(B_{kr})\lesssim E(B_{r}).    \label{jk}
\end{align}
Focusing on \eqref{nameT}, we find
\begin{align}d^\star \vec{T}= \nabla^j \vec{T}_j=\Delta\vec{H} +2\nabla^j(
Hh_{jk}\nabla^k\vec{\Phi}-H^2\nabla_j\vec{\Phi}) . \label{nied}
\end{align}

\noindent
Using \eqref{nied} and introducing Remark \ref{rama}  and Corollary \ref{takeova} into the return equation yields for all $k\in (0,1/2)$ 
\begin{align}
\norm{\Delta H}_{L^{4/3}(B_{kr})}\lesssim ||d^\star \vec{T}||_{L^{4/3}(B_{kr})} +\varepsilon E(B_r) \lesssim (k+\varepsilon) E(B_r).   \label{jk1}
\end{align}

\noindent
On the other hand, a direct computation reveals
$$(d\vec{T})_{ij}=2(h_{ik}\nabla_jH- h_{jk}\nabla_iH)\nabla^k\vec{\Phi} +2\nabla_jH^2\nabla_i\vec{\Phi}-2\nabla_iH^2\nabla_j\vec{\Phi},$$

\noindent
so that
\begin{gather}
||d\vec{T}||_{L^{4/3}(B_{kr})}  \lesssim \norm{h}_{L^4(B_{kr})}  \norm{dH}_{L^2(B_{kr})}   \lesssim \varepsilon E(B_r).   \label{jk2}
\end{gather}

\noindent
Using \eqref{jk}, \eqref{jk1} and \eqref{jk2} into Lemma \ref{tech} yields
\begin{align}
||\vec{T}||_{L^2(B_{kr})} \lesssim (\varepsilon +k)E(B_r). \label{hg1}
\end{align}

\noindent
We now find
\begin{align}
\norm{dH}_{L^2(B_{kr})} &\overset{\eqref{nameT}}\lesssim ||\vec{T}||_{L^2(B_{kr})} +\norm{H}_{L^4(B_{kr})} \norm{h}_{L^4(B_{kr})}       \nonumber
\\&\overset{\eqref{hg1}}\lesssim (\varepsilon+k) E(B_r).  \label{ja1b}
\end{align}

\noindent
By the Sobolev embedding theorem and the Jensen's inequality, it follows that
\begin{align}
\norm{H}_{L^4(B_{kr})} \lesssim \norm{H}_{L^2(B_{kr})} +\norm{dH}_{L^2(B_{kr})}  \overset {\eqref{ja1b}} \lesssim (\varepsilon +k) E(B_r).  \label{233}
\end{align}
 Combining \eqref{ja1b} and \eqref{233}, we have the crucial estimate
$$E(B_{kr})\equiv \norm{dH}_{L^2(B_{kr})}+\norm{H}_{L^4(B_{kr})}  \lesssim (\varepsilon +k)E(B_r).$$

\noindent
Since $\varepsilon$ and $k$ may be chosen as small as we please, by a standard controlled growth argument (cf. Lemma 5.13 in \cite{gia}), we  find the Morrey decay
\begin{align}
E(B_r)\equiv \norm{DH}_{L^2(B_r)} +\norm{H}_{L^4(B_r)} \lesssim E(B_1)r^\beta\quad\quad\forall\,\, r<1.    \label{morrey}
\end{align}
We emphasise that this is true for any $\beta \in (0,1)$.

Consider the maximal function

$$\mathcal{M}_{3-\beta}[f]:= \sup_{\rho>0}\rho^{-1-\beta}\norm{f}_{L^1(B_\rho)}.$$

\noindent
By Jensen's inequality, we have 
$$\mathcal{M}_{3-\beta}[f]\lesssim \sup_{\rho>0} \rho^{-\beta}\norm{f}_{L^{4/3}(B_\rho)}.$$
Using the Morrey decay  \eqref{morrey} and \eqref{jk}, we obtain
\begin{align}
\lVert\mathcal{M}_{3-\beta}[\Delta\vec{H}]\rVert_{L^\infty(B_r)} \lesssim E(B_1)\quad\quad\forall\,\, r<1.  
\end{align}

\noindent
Let $f$ be a locally integrable function on $\mathbb{R}^n$.
Recall that for a number $\alpha$ satisfying $\alpha\in (0, n)$, the {\it Riesz potential $\mathcal{I}_\alpha$ of order $\alpha$ } of $f$ is defined by the convolution
$$(\mathcal{I}_\alpha *f)(x):=\frac{1}{C(\alpha)}\int_{\mathbb{R}^n} f(y)|x-y|^{\alpha-n} dy$$
where $C(\alpha)=\pi^{n/2}2^\alpha\,\, \Gamma(\frac{\alpha}{2})/\Gamma(\frac{n-\alpha}{2})$ is a constant.

 \noindent
We will now use the following result from \cite{ada}.  
\begin{prop}  \label{nbd}
If $\alpha>0$, $0<\lambda\leq n$, $1<p<\lambda/\alpha$, $1\leq q\leq \infty$, and $f\in L^p(\mathbb{R}^n)$ with $\mathcal{M}_{\lambda/p} [f]\in L^q(\Omega)$, $\Omega\subset \mathbb{R}^n$, then
$$\norm{\mathcal{I}_\alpha[f]}_{L^r(\Omega)} \lesssim \norm{M_{\lambda/p }[f]}_{L^q(\Omega)}^{\alpha p/\lambda} \norm{f}_p^{1-\alpha p/\lambda} $$
where $1/r=1/p-\alpha/\lambda+(\alpha p)/(\lambda q)$.
\end{prop}

\vskip3mm
\noindent
Putting $\alpha=1$, $q=\infty$, $p=4/3$, $\lambda=4(3-\beta)/3$ in Proposition \ref{nbd}, we find
\begin{align}
\norm{\mathcal{I}_1[\Delta \vec{H}]}_{L^s(B_r)}& \lesssim \norm{\mathcal{M}_{3-\beta}[\Delta\vec{H}]}_{L^{\infty}(B_r)}^{1/(3-\beta)}  \norm{\Delta \vec{H}}_{L^{4/3}(B_r)}^{(2-\beta)/(3-\beta)}   \nonumber
\\&\lesssim E(B_1)^{1-\frac{4}{3s}} E(B_r)^{\frac{4}{3s}}
\end{align}

\noindent
where
$$s:= \frac{4}{3}\left( \frac{3-\beta}{2-\beta} \right)\in \left(2, \frac{8}{3} \right)$$

\noindent
and $\mathcal{I}_1$ denotes the Riesz potential of order 1.

\noindent
We decompose the mean curvature vector  $\vec{H}=\vec{H}_0+\vec{H}_1$ such that
\begin{equation*}
 \begin{cases}
          \Delta \vec{H}_0=\vec{0} & \mbox{in}\,\, B_{1} \\
            \,\,\,\,\, \vec{H}_0= \vec{H} & \mbox{on} \,\, \partial B_{1},
       \end{cases} \quad
\quad\quad\mbox{and} \quad\quad\begin{cases}
            \Delta \vec{H}_1=\mathcal{I}_1[\Delta \vec{H}] & \mbox{in}\,\, B_{1}\\
            \,\,\,\,\, \vec{H}_1= \vec{0} & \mbox{on} \,\, \partial B_{1}.
       \end{cases}
\end{equation*}

\noindent
Let now $s\in\left(2,\frac{8}{3} \right)$. Using standard elliptic estimates and the Sobolev embedding theorem, we have for all $r<1$
\begin{align}
\|D\vec{H}\|_{L^s(B_r)}& \lesssim \|D\vec{H}_0  \|_{L^s(B_r)} +\|D\vec{H}_1  \|_{L^s(B_r)}  \nonumber
\\&\lesssim r^{\frac{4}{s}-2} \|D\vec{H}_0  \|_{L^2(B_1)} +\|D\vec{H}_1 \|_{L^s(B_r)}  \nonumber
\\&\lesssim  r^{\frac{4}{s}-2}\|D\vec{H}\|_{L^2(B_1)}+  \|\mathcal{I}_1[\Delta\vec{H}]\|_{L^s(B_r)}\nonumber
\\&  \lesssim r^{\frac{4}{s}-2}E(B_1) +E(B_1)^{1-\frac{4}{3s}} E(B_r)^{\frac{4}{3s}} \nonumber
\\&\lesssim (r^{\frac{4}{s}-2} +1)E(B_1)  .  \label{gha}
\end{align}

\noindent
As \eqref{gha} holds for all $s\in \left(2,\frac{8}{3}   \right)$, we see that the integrability of $D\vec{H}$ has been improved.

\vskip 3mm
\noindent
With this new information on the integrability of $d\vec{H}$, the above procedure may be repeated untill we obtain that $\vec{H}$ is Lipschitz. Once this is known, we see that $\Delta \vec{\Phi}$ is as well Lipschitz. Since $|g|^{1/2}\in W^{1,4}$ by hypothesis, we see that $|g|^{1/2}\Delta \vec{\Phi}$ lies in $W^{1,4}$. We have
$$\mathcal{L}[\vec{\Phi}]=|g|^{1/2}\Delta\vec{\Phi}.$$
Calling upon Theorem 1.1 in \cite{gru}, the Green kernel $\mathcal{G}$ of $\mathcal{L}$ satisfies $D\mathcal{G}\in L^{{4/3},\infty}$ where $L^{{4/3},\infty}$ is the \footnote{The weak-$L^{4/3}$ Marcinkiewicz space $L^{{4/3},\infty} (B_1(0))$ is defined as those functions $f$ which satisfy $\sup_{\alpha >0} \alpha^{4/3}|\{x\in B_1(0): |f(x)|\geq \alpha \}|<\infty$ . The space $L^{{4/3},\infty}$ is also a Lorentz space.}weak Marcinkiewicz space.

\noindent
Formally, the solution $\vec{\Phi}$
is given by the convolution
$$\vec{\Phi}=\mathcal{G}* \mathcal{L}[\vec{\Phi}].$$

\noindent
Hence by the convolution rule for Lorentz spaces (cf. \cite{hunt})
$$D^2\vec{\Phi}= D\mathcal{G}*D\mathcal{L}[\vec{\Phi}]\in L^{{4/3},\infty}*L^4\subset L^p\quad\quad\forall\,\, p<\infty.$$

\noindent
Accordingly, by the Sobolev embedding theorem, $D\vec{\Phi}\in \bigcap_{p<\infty}W^{1,p} \subset \bigcap_{\alpha<1} C^{0,\alpha}$. The regularity of $\vec{\Phi}$ has thus also improved. In particular, the metric coefficients are H\"{o}lder continuous of all orders, and it follows that the standard analysis of second-order uniformly elliptic  operators is now at hand (cf. \cite{gil}). Eventually, by standard elliptic arguments, we reach the conclusion that $\vec{\Phi}$ is smooth. 

\vskip 3mm
Finally, by standard elliptic arguments we obtain that
\begin{align}
\norm{DH}_{L^s(B_r)} +\frac{1}{r} \norm{H}_{L^s(B_r)}  \lesssim \frac{1}{r^2} \left(\norm{DH}_{L^2(B_1)} +\norm{H}_{L^4(B_1)}   \right)\quad\forall \, r<1, s>2.
\end{align}
The result then follows as $s\nearrow \infty$. This concludes the proof.

\appendix

\section{Appendix}
\subsection{Auxiliary computations } \label{grace5}
\noindent
Let $\vec{A}$ and  $\vec{B}$ be 2-vectors. From the definition of the first-order contraction operator (see page \pageref{se123} for definition), we have  the following multiplication rule.
\begin{align}
\vec{A}\cdot \vec{B}&:= (\vec{A}_1\wedge\vec{A}_2) \cdot (\vec{B}_1\wedge\vec{B}_2)   \nonumber
\\&= (\vec{A}_1\cdot\vec{B}_1)(\vec{A}_2\cdot\vec{B}_2) -(\vec{A}_1\cdot\vec{B}_2)(\vec{A}_2\cdot\vec{B}_1) , \nonumber
\end{align}
\begin{align}
\vec{A}\bullet \vec{B}&:=(\vec{A}_1\wedge\vec{A}_2) \bullet (\vec{B}_1\wedge\vec{B}_2)   \nonumber
\\&=(\vec{A}_2\cdot \vec{B}_2)\vec{A}_1\wedge \vec{B}_1- (\vec{A}_2\cdot \vec{B}_1)\vec{A}_1\wedge \vec{B}_2-(\vec{A}_1\cdot \vec{B}_2)\vec{A}_2\wedge \vec{B}_1\nonumber
\\&\quad + (\vec{A}_1\cdot \vec{B}_1)\vec{A}_2\wedge \vec{B}_4.\nonumber
\end{align}

\noindent
With the rules above, we are ready for the following computation. We know that $S_0$ and $\vec{R}_0$ satisfy

$$dS_0=\vec{L}\overset{\cdot} \wedge_4 d\vec{\Phi}\quad\quad\mbox{and}\quad\quad d\vec{R}_0=\vec{L}\overset{\wedge}\wedge_4 d\vec{\Phi}+\star d\vec{u} -\vec{J}_0\wedge \star d\vec{\Phi} .$$

\noindent
We compute
\begin{align}
&\quad\star \vec{\eta}\overset{\bullet}\wedge_4 \star\left(d\vec{R}_0-\star d\vec{u} +\vec{J}_0\wedge\star d\vec{\Phi}  \right)  \nonumber
\\&=\star\vec{\eta}\overset{\bullet}\wedge_4 \star(\vec{L}\overset{\wedge}\wedge_4 d\vec{\Phi})  \nonumber
\\&=    \frac{1}{3!} (\star\vec{\eta})_{[ab}\bullet (\star(\vec{L}\overset{\wedge}\wedge_4 d\vec{\Phi}))_{r]} \,\, dx^{a}\wedge_4 dx^b\wedge_4 dx^r      \nonumber
\\&= \frac{1}{3!}\frac{1}{6}\delta_{abr}^{\alpha\beta\gamma}\left (\epsilon_{\alpha\beta pq}\frac{1}{2!}\nabla^p\vec{\Phi}\wedge\nabla^q\vec{\Phi} \right) \bullet \left( \epsilon_{\gamma ijk} \frac{1}{3!} \vec{L}^{ij}\wedge \nabla^k \vec{\Phi} \right) \,\, dx^{a}\wedge_4 dx^b\wedge_4 dx^r   \nonumber
\\&=\frac{1}{216}\delta^{\alpha\beta\gamma}_{abr}\epsilon_{\alpha\beta pq}\epsilon_{\gamma ijk}\left( -g^{qk} \vec{L}^{ij}\wedge\nabla^p\vec{\Phi} -(\vec{L}^{ij}\cdot\nabla^{p}\vec{\Phi})\nabla^k\vec{\Phi}\wedge\nabla^q\vec{\Phi}     \right) \,\,   dx^{a}\wedge_4 dx^b\wedge_4 dx^r    \nonumber
\\&:= \vec{A} -\frac{1}{126}\delta^{\alpha\beta\gamma}_{abr}\epsilon_{\alpha\beta pq}\epsilon_{\gamma ijk} \vec{B}^{ijpk}\wedge\nabla^q\vec{\Phi}\,\, dx^{a}\wedge_4 dx^b\wedge_4 dx^r. \label{derc}
\end{align}

\noindent
We have for the first term of \eqref{derc} 
\begin{align}
\vec{A}&= -\frac{1}{216} \delta^{\alpha\beta\gamma}_{abr} \epsilon_{\alpha\beta pq}\tensor{\epsilon}{_{\gamma ij}^{q}}\vec{L}^{ij}\wedge_4 \nabla^p\vec{\Phi}\,\, dx^{a}\wedge_4 dx^b\wedge_4 dx^r    \nonumber
\\&=-\frac{1}{216} \delta^{\alpha\beta\gamma}_{abr}\epsilon_{\alpha\beta pq}\epsilon^{cijq}g_{c\gamma}\vec{L}_{ij}\wedge \nabla^p\vec{\Phi} \,\, dx^{a}\wedge_4 dx^b\wedge_4 dx^r     \nonumber
\\&= -\frac{1}{216}\delta^{\alpha\beta\gamma}_{abr} \delta_{\alpha\beta pq}^{cijq}g_{c\gamma} \vec{L}_{ij}\wedge \nabla^p\vec{\Phi} \,\, dx^{a}\wedge_4 dx^b\wedge_4 dx^r   \nonumber
\\&= -\frac{1}{12} \vec{L}_{[ab}\wedge \nabla_{r]}\vec{\Phi}\,\, dx^{a}\wedge_4 dx^b\wedge_4 dx^r   \nonumber
\\&=-\vec{L}\overset{\wedge} \wedge_4 d\vec{\Phi}.  \label{plpl}
\end{align}

\noindent
Next is to compute the second term of \eqref{derc}. Note that we have set $\vec{B}^{ijkp}:= (\vec{L}^{ij}\cdot\nabla^p\vec{\Phi})\nabla^k\vec{\Phi}$ where $\vec{B}$ is antisymmetric in indices $(i,j,k)$.  
\\
We have
\begin{align}
&\quad\frac{1}{216}\delta^{\alpha\beta\gamma}_{abr} \epsilon_{\alpha\beta pq}\epsilon_{\gamma ijk} \vec{B}^{ijpk}\wedge\nabla^q\vec{\Phi}\,\, dx^{a}\wedge_4 dx^b\wedge_4 dx^r   \nonumber
\\&=  \frac{1}{216}\delta^{\alpha\beta \gamma}_{abr}\epsilon^{lm pq}g_{l \alpha}g_{m \beta}\epsilon_{\gamma ijk} \tensor {\vec{B}}{^{ij}_{p}^k}\wedge \nabla_q\vec{\Phi}   \,\, dx^{a}\wedge_4 dx^b\wedge_4 dx^r    \nonumber
\\&=\frac{1}{216}\delta^{\alpha\beta\gamma}_{abr}\delta^{lm pq}_{\gamma ijk}g_{l \alpha}g_{m \beta} \tensor {\vec{B}}{^{ij}_{p}^k} \wedge\nabla_q\vec{\Phi}  \,\, dx^{a}\wedge_4 dx^b\wedge_4 dx^r   \nonumber
\\&=\frac{1}{216}  \delta^{\alpha\beta\gamma}_{abr}   \left( \delta^p_\gamma \delta^{lm q}_{ijk} -\delta^q_\gamma \delta^{lm p}_{ijk} \right)      g_{l \alpha}g_{m \beta} \tensor {\vec{B}}{^{ij}_{p}^k} \wedge\nabla_q\vec{\Phi}  \,\,                     dx^{a}\wedge_4 dx^b\wedge_4 dx^r   \nonumber
\\&= \frac{1}{216}\delta^{\alpha\beta\gamma}_{abr} \left[\delta^p_\gamma\left(\delta^l_i \delta^{m q}_{jk} -\delta^l_j \delta^{m q}_{ik} + \delta^l_k \delta^{m q}_{ij} \right)  -\delta^q_\gamma \left( \delta^l_i \delta^{m p}_{jk} -\delta^l_j \delta^{m p}_{ik} \right.\right.  \nonumber
\\&\left.\left.\quad\quad\quad\quad\quad\quad
+\delta^l_k \delta^{m p}_{ij}  \right)\right]   g_{l \alpha}g_{m \beta} \tensor {\vec{B}}{^{ij}_{p}^k} \wedge\nabla_q\vec{\Phi}  \,\,                     dx^{a}\wedge_4 dx^b\wedge_4 dx^r    \nonumber
\\&=  \frac{1}{216}\delta^{\alpha\beta\gamma}_{abr}\left[2\delta^p_\gamma \delta^l_i\delta^{m q}_{jk}+ \delta^p_\gamma \delta^l_k\delta^{m q}_{ij} -2\delta^q_\gamma \delta^l_i\delta^{m p}_{jk}-\delta^q_\gamma \delta^l_k\delta^{m p}_{ij}\right]     g_{l \alpha}g_{m \beta} \tensor {\vec{B}}{^{ij}_{p}^k} \wedge\nabla_q\vec{\Phi}  \,\,                     dx^{a}\wedge_4 dx^b\wedge_4 dx^r   \nonumber
\\&=  \frac{1}{216}\delta^{\alpha\beta\gamma}_{abr}\left[2\delta^p_\gamma \delta^l_i \delta^m_j \delta^q_k -  2\delta^p_\gamma \delta^l_i \delta^m_k \delta^q_j +\delta^p_\gamma \delta^l_k \delta^m_i \delta^q_j -\delta^p_\gamma \delta^l_k \delta^m_j \delta^q_i \right]           g_{l \alpha}g_{m \beta} \tensor {\vec{B}}{^{ij}_{p}^k} \wedge\nabla_q\vec{\Phi}  \,\,                   dx^{a}\wedge_4 dx^b\wedge_4 dx^r   \nonumber
\\&\quad+\frac{1}{216}\left[-2\delta^q_\gamma \delta^l_i \delta^m_j \delta^p_k +  2\delta^q_\gamma \delta^l_i \delta^m_k \delta^p_j -\delta^q_\gamma \delta^l_k \delta^m_i \delta^p_j +\delta^q_\gamma \delta^l_k \delta^m_j \delta^p_i \right]           g_{l \alpha}g_{m \beta} \tensor {\vec{B}}{^{ij}_{p}^k} \wedge\nabla_q\vec{\Phi}  \,\,                    dx^{a}\wedge_4 dx^b\wedge_4 dx^r         \nonumber
\\&=   \frac{1}{216}\delta^{\alpha\beta\gamma}_{abr}\left[2\delta^q_\gamma\delta^l_k \delta^m_j \delta^p_i  -2\delta^q_\gamma\delta^l_i \delta^m_j \delta^p_k  +2\delta^q_\gamma\delta^l_i \delta^m_k \delta^p_j   \right]                         g_{l \alpha}g_{m \beta} \tensor {\vec{B}}{^{ij}_{p}^k} \wedge\nabla_q\vec{\Phi}  \,\,                     dx^{a}\wedge_4 dx^b\wedge_4 dx^r          \nonumber
\\&= \frac{1}{36} \delta^{\alpha\beta\gamma}_{abr} \delta^q_\gamma\delta^l_i \delta^m_j \delta^p_k      g_{l \alpha}g_{m \beta} \tensor {\vec{B}}{^{ij}_{p}^k} \wedge\nabla_q\vec{\Phi}  \,\,                     dx^{a}\wedge_4 dx^b\wedge_4 dx^r      \nonumber
\\&= \frac{1}{36} \delta^{\alpha\beta\gamma}_{abr} \delta^k_p\,\,\,\,  \tensor {\vec{B}}{_{\alpha\beta}^{p}_k} \wedge\nabla_\gamma\vec{\Phi}  \,\,                     dx^{a}\wedge_4 dx^b\wedge_4 dx^r   . \label{plpl1}
\end{align}
where $\tensor {\vec{B}}{_{\alpha\beta}^{p}_k} $ is antisymmetric in the indices $\alpha,\beta$ and $k$.

\noindent
Also, we have
\begin{align}
\star\vec{\eta} \overset{\bullet}\wedge_4 \star(\vec{J}_0\wedge \star d\vec{\Phi})&=-\frac{1}{3!2!}\frac{1}{6}\delta_{abr}^{\alpha\beta\gamma}\epsilon_{\alpha\beta pq} (\nabla^p\vec{\Phi}\wedge \nabla^q\vec{\Phi}) \bullet (\vec{J}_0\wedge\nabla_\gamma\vec{\Phi})\,\, \,\,                     dx^{a}\wedge_4 dx^b\wedge_4 dx^r  \nonumber
\\&=  \frac{1}{36}\delta^{\alpha\beta\gamma}_{abr}\epsilon_{\alpha\beta pq} \delta^{q}_\gamma \nabla_p\vec{\Phi} \wedge\vec{J}_0 \,\, dx^{a}\wedge_4 dx^b\wedge_4 dx^r \nonumber
\\&= -\vec{J}_0\wedge_4\star d\vec{\Phi} . \label{plpl2}
\end{align}

\noindent
Using \eqref{plpl}, \eqref{plpl1} and \eqref{plpl2} in \eqref{derc}, we have
$$-d\vec{R}_0=\star\vec{\eta}\overset{\bullet}\wedge_4 \star d\vec{R}_0 +\star d\vec{u}+ \star\vec{\eta}\overset{\bullet}\wedge_4 d\vec{u} + \frac{1}{36} \delta^{\alpha\beta\gamma}_{abr} \delta^k_p\,\,\,\,  \tensor {\vec{B}}{_{\alpha\beta}^{p}_k} \wedge\nabla_\gamma\vec{\Phi}  \,\,                     dx^{a}\wedge_4 dx^b\wedge_4 dx^r     .   $$
\noindent
On the other hand, we have

\begin{align}
\star \vec{\eta} \wedge_4 \star dS&= \frac{1}{432} \delta^{\alpha\beta\gamma}_{abr}\epsilon_{\alpha\beta pq}\epsilon_{\gamma ijk} (\vec{L}^{ij}\cdot\nabla^k\vec{\Phi})(\nabla^p\vec{\Phi}\wedge \nabla^q\vec{\Phi})\,\, \,\,                     dx^{a}\wedge_4 dx^b\wedge_4 dx^r          \nonumber
\\&= \frac{1}{432} \delta^{\alpha\beta\gamma}_{abr}\delta^{l    m    pq}_{\gamma ijk} g_{l     \alpha}g_{m    \beta}  \tensor {\vec{B}}{^{ijk}_{p}} \wedge\nabla_q\vec{\Phi}  \,\,                   \,\,                     dx^{a}\wedge_4 dx^b\wedge_4 dx^r   \nonumber
\\&=   \frac{1}{432}\delta^{\alpha\beta\gamma}_{abr} \left(\delta^p_\gamma \delta^{l    m    q}_{ijk}-\delta^q_\alpha \delta^{l    m    p}_{ijk}  \right) g_{l     \alpha}g_{m    \beta}  \tensor {\vec{B}}{^{ijk}_{p}} \wedge\nabla_q\vec{\Phi}  \,\,                    \,\,                     dx^{a}\wedge_4 dx^b\wedge_4 dx^r     \nonumber
\\&= \frac{1}{72}\delta^{\alpha\beta\gamma}_{abr} \left( g_{i \alpha}g_{j \beta}  \tensor {\vec{B}}{^{ijk}_{\gamma}} \wedge\nabla_k\vec{\Phi} -  g_{i \alpha}g_{j \beta}\tensor {\vec{B}}{^{ijk}_{k}} \wedge\nabla_\gamma\vec{\Phi} \right)\,\,                     \,\,                     dx^{a}\wedge_4 dx^b\wedge_4 dx^r     \nonumber
\\&=-\frac{1}{36}\delta^{\alpha\beta\gamma}_{abr} g_{i \alpha}g_{j \beta} \tensor {\vec{B}}{^{ijk}_{k}} \wedge\nabla_\gamma\vec{\Phi} \,\,                     \,\,                     dx^{a}\wedge_4 dx^b\wedge_4 dx^r  \nonumber
\\&=-\frac{1}{36}\delta^{\alpha\beta\gamma}_{abr} \delta^k_p \tensor {\vec{B}}{_{\alpha\beta k}^{p}} \wedge\nabla_\gamma \vec{\Phi} \,\,                     \,\,                     dx^{a}\wedge_4 dx^b\wedge_4 dx^r   \nonumber
\end{align}
where $ \tensor {\vec{B}}{_{\alpha\beta k}^{p}}$ is antisymmetric in the indices $\alpha, \beta$ and $k$.

\noindent
Thus we arrive at
$$d\vec{R}_0=\star\vec{\eta}\overset{\bullet}\wedge_4 \star d\vec{R}_0 -\star d\vec{u}- \star\vec{\eta}\overset{\bullet}\wedge_4 d\vec{u} +\star\vec{\eta}\wedge_4\star dS_0\,\,   +\vec{Q}                    $$

\noindent
where
$$\vec{Q}:= \frac{1}{12}\delta_{abr}^{\alpha\beta\gamma}( \tensor{\vec{B}}{_{k\alpha}^k_\beta}-\tensor{\vec{B}}{_{k\alpha\beta}^k})\wedge\nabla_\gamma \vec{\Phi}      \,\,                     dx^{a}\wedge_4 dx^b\wedge_4 dx^r. $$

\subsection{Proof of Claims \ref{cll} and \ref{claimclaim}}

\textbf{Claim} \ref{cll}.
{\it The coefficients $a^{ij}$ belong to the space $W^{2,2}\cap L^\infty$. }

\begin{proof}
First, we show that $W^{2,2}\cap L^\infty$ is an algebra under pointwise multiplication. Let $\alpha,\beta \in W^{2,2}\cap L^\infty$. Clearly, $\alpha \beta \in L^\infty$. It remains to show that $\alpha\beta \in W^{2,2}$. By Liebnitz rule we write
\begin{align}
 D ^2(\alpha \beta)=  \alpha D ^2 \beta+ \beta D ^2 \alpha+ 2 D  \alpha D  \beta . \nonumber
\end{align}
Now, $\alpha\in L^\infty,  D ^2 \beta\in L^2$ implies that $\alpha D ^2 \beta\in L^2$. By the same token $\beta D ^2\alpha\in L^2$. By Sobolev embedding, the functions $ D  \alpha, D  \beta\in W^{1,2}\subset L^4$ so that the product $ D  \alpha D  \beta\in L^2$. We arrive at $\alpha\beta\in W^{2,2}$. Thus $\alpha\beta\in W^{2,2}\cap L^\infty$.

\noindent
Next, using the hypothesis $\vec{\Phi}\in W^{3,2}\cap W^{1,\infty}$ and the fact that $W^{2,2}\cap L^\infty$ is an algebra under pointwise multiplication, we have
$g^{ij}:=\nabla^i\vec{\Phi}\cdot \nabla^j\vec{\Phi}\in W^{2,2}\cap L^\infty$. Hence $a^{ij}\in W^{2,2}\cap L^\infty$.
\end{proof}

\noindent
\textnormal{ \textbf{Claim} \ref{claimclaim}.}
 Let $\alpha$ be a function in $(L^\infty \cap W^{2,2})$. Let $\beta$ be a function in  $\dot{W}^{-1,2}$. Then the product $\alpha\beta$ belongs to $  \dot{W}^{-1,2} \oplus L^{4/3} \subset \dot{W}^{-1,2}$.

\begin{proof}
Since  $\beta \in \dot{W}^{-1,2}$, there exists $\gamma\in L^2$ such that $ D  \gamma=\beta$. We have
\begin{align}
\alpha\beta=  D (\alpha\gamma)-\gamma  D  \alpha.  \nonumber
\end{align}
Now, $\alpha\in L^\infty$ and $\gamma\in L^2$ implies that $ D (\alpha\gamma)\in \dot{W}^{-1,2}$. By the Sobolev embedding theorem, $\alpha\in W^{2,2}\subset W^{1,4}$ so that $ D  \alpha\in L^4$ and the product $\gamma D \alpha \in L^{4/3}$. By the dual of the Sobolev embedding theorem, $\alpha\beta\in  \dot{W}^{-1,2}\oplus L^{4/3}  \subset \dot{W}^{-1,2}.$
\end{proof}

\end{document}